\documentclass[
,1p,times,authoryear]{elsarticle}
\usepackage{amsmath}
\usepackage{amsthm}
\usepackage{amssymb}
\usepackage{amsfonts}

\usepackage[utf8]{inputenc}
\usepackage[T1]{fontenc}

\usepackage{xspace}
\usepackage[dvipsnames]{xcolor}
\usepackage{tikz}
\usepackage{graphicx}
\usepackage{mathtools}
\usepackage{cleveref}
\usepackage{nameref} 
\usepackage[ruled,linesnumbered]{algorithm2e}
\usepackage{geometry,pdflscape}

\newtheorem{theorem}{Theorem}
\newtheorem{lemma}[theorem]{Lemma}
\newtheorem{corollary}[theorem]{Corollary}

\newtheorem{definition}[theorem]{Definition}
\newtheorem{notation}[theorem]{Notation}
\newtheorem{example}[theorem]{Example}
\newtheorem{remark}[theorem]{Remark}

\newtheorem{intuition}[theorem]{Intuition}

\newcommand{\C}{\mathbb{C}} 
\newcommand{\R}{\mathbb{R}} 
 
\newcommand{\K}{\mathbb{K}} 
\newcommand{\N}{\mathbb{N}} 
 
\newcommand{\RLC}{\mathfrak{R}} 
\newcommand{\CLC}{\mathfrak{C}} 

\newcommand{\eg}{\emph{e.g.}\xspace}
\newcommand{\ie}{\emph{i.e.}\xspace}
 
\newcommand{\RP}{\mathbb{R}\mathbb{P}}
\newcommand{\CP}{\mathbb{C}\mathbb{P}}
\newcommand{\CsP}{\mathbb{C^*}\mathbb{P}}
\newcommand{\Smz}{\setminus \{0\}}

\newcommand{\PC}{\mathcal{P}} 
\newcommand{\LC}{\mathcal{L}} 
\newcommand{\CC}{\mathcal{C}} 
\newcommand{\bbar}{\; | \;}
\newcommand{\Ks}{\mathbb{K^*}}
\newcommand{\Rs}{\mathbb{R^*}}
\newcommand{\Cs}{\mathbb{C^*}}
\newcommand{\As}{{A^*}}

\newcommand{\Lim}{\mathbb{L}} 

\DeclareMathOperator{\hal}{\mathbf{hal}}
\DeclareMathOperator{\gal}{\mathbf{gal}}
\DeclareMathOperator{\inter}{\mathbf{int}}
\DeclareMathOperator{\sh}{{\mathbf{sh}}}
\newcommand*\conj[1]{\overline{#1}}
\newcommand{\norm}[1]{\left\lVert#1\right\rVert}
\newcommand{\I}{\mathbb{I}} 

\newcommand{\ii}{\mathrm{i}}
\newcommand{\II}{\mathrm{I}}
\newcommand{\JJ}{\mathrm{J}}

\newcommand{\KsP}{\mathbb{K^*}\mathbb{P}}

\newcommand{\KP}{\mathbb{K}\mathbb{P}}
\DeclareMathOperator{\psh}{\mathbf{psh}}
\DeclarePairedDelimiter\abs{\lvert}{\rvert}%

\makeatletter
    \def\page{p.\ }
    \def\Dx{{\Delta x}}
    
    \def\Dt{{\Delta t}}
    \def\Dv{{\Delta v}}

\makeatother

\DeclareMathOperator*{\argmax}{arg\,max}
 
\newcommand{\dd}{\, \mathrm d}

\newcommand{\dt}{\, \mathrm dt} 
\newcommand{\A}{\mathbb{A}} 
 
\newcommand{\V}{\mathbb{V}}

\newcommand{\jrguk}{Kortenkamp and Richter-Gebert\xspace}

\newcommand*{\fullref}[1]{\hyperref[{#1}]{\Cref*{#1} (\nameref*{#1})}} 

\newcommand{\join}{\textbf{join}} 
\newcommand{\meet}{\textbf{meet}}

\newcommand{\normsymb}{\norm{\cdot}}
\newcommand{\normalize}[1]{#1_{\normsymb}}

\newcommand\drawMcirc{
    \begin{tikzpicture}
    \draw (0,0) circle (1cm);
    \draw (1,0) circle (1cm);
    \draw[color=blue!80!black, line width = 0.5pt] (0.5,-2) -- (0.5,2);
    
    \fill[color=red] (0.5,0.866) circle (2pt);
    \fill[color=green] (0.5,-0.866) circle (2pt);
    \end{tikzpicture}
    \quad \quad
    \begin{tikzpicture}
    \draw (-1,0) circle (1cm);
    \draw (1,0) circle (1cm);
    \draw[color=blue!80!black,dashed,  line width = 0.5pt] (0,-2) -- (0,2);
    
    \fill[color=green] (0,0) circle (2pt); 
    \fill[color=red] (0,0) -- (180:2pt) arc (180:0:2pt) -- cycle; 
    \end{tikzpicture}
    \quad \quad
        \begin{tikzpicture}
    \draw (0,0) circle (1cm);
    \draw (2.5,0) circle (1cm);
    \draw[color=blue!80!black, line width = 0.5pt] (1.25,-2) -- (1.25,2);
    \end{tikzpicture}
}
\newcommand\stablecircles{
   \begin{tikzpicture}
\draw[color=DarkOrchid] (0,0) circle (1cm);
\draw[color=Emerald]  (1.5,0) circle (1cm);
\draw[color=blue!80!black, line width = 0.5pt] (0.75,-1.5) -- (0.75,1.5);

\fill[] (0,0) circle (1pt); 
\fill[] (1.5,0) circle (1pt); 
\draw[line width = 0.5pt] (-1.5,0) -- (3,0);

\fill[color=red] (0.75,0.66) circle (2pt);
\fill[color=green] (0.75,-0.66) circle (2pt);
\end{tikzpicture} 
\begin{tikzpicture}
\draw[] (0,0) circle (1cm);
\draw[color=blue!80!black, line width = 0.5pt,dashed] (0,-1.5) -- (0,1.5);

\fill[] (0,0) circle (1pt); 
\draw[line width = 0.5pt] (-1.5,0) -- (1.5,0);

\fill[color=red,opacity=0.5] (0,1) circle (2pt);
\fill[color=green,opacity=0.5] (0,-1) circle (2pt);
\end{tikzpicture}
   \begin{tikzpicture}
\draw[color=DarkOrchid] (1.5,0) circle (1cm);
\draw[color=Emerald]  (0,0) circle (1cm);
\draw[color=blue!80!black, line width = 0.5pt] (0.75,-1.5) -- (0.75,1.5);

\fill[] (0,0) circle (1pt); 
\fill[] (1.5,0) circle (1pt); 
\draw[line width = 0.5pt] (-1.5,0) -- (3,0);

\fill[color=red] (0.75,0.66) circle (2pt);
\fill[color=green] (0.75,-0.66) circle (2pt);
\end{tikzpicture} 
}
\newcommand\unstablecircles{
   \begin{tikzpicture}[baseline]
\draw[] (0,0) circle (1cm);
\draw[color=DarkOrchid] (0.1,-0.1) circle (1cm);
\draw[color=Emerald] (0,0) circle (1cm);
\draw[color=blue!80!black, line width = 0.5pt,shorten <=-0.3cm, shorten >=-0.3cm] (0.75,0.65) -- (-0.65,-0.75);

\fill[color=red] (0.75,0.65) circle (2pt);
\fill[color=green] (-0.65,-0.75) circle (2pt);
\end{tikzpicture} 
\quad
\quad
    \begin{tikzpicture}[baseline]
\draw[] (0,0) circle (1cm);
\draw[color=blue!80!black, line width = 0.5pt,dashed] (0,-1.3) -- (0,1.3);

\fill[color=red,opacity=0.5] (0,1) circle (2pt);
\fill[color=green,opacity=0.5] (0,-1) circle (2pt);
\end{tikzpicture}
\quad
\quad
\begin{tikzpicture}[baseline]
\draw[] (0,0) circle (1cm);
\draw[color=DarkOrchid] (0.1,0.1) circle (1cm);
\draw[color=Emerald] (0,0) circle (1cm);
\draw[color=blue!80!black, line width = 0.5pt,shorten <=-0.3cm, shorten >=-0.3cm] (-0.65,0.75) -- (0.75,-0.65);

\fill[color=red] (-0.65,0.75) circle (2pt);
\fill[color=green] (0.75,-0.65) circle (2pt);
\end{tikzpicture}
\quad
    \begin{tikzpicture}[baseline]
\draw[] (0,0) circle (1cm);
\draw[color=blue!80!black, line width = 0.5pt,dashed] (-1.3,0) -- (1.3,0);

\fill[color=green,opacity=0.5] (1,0) circle (2pt);
\fill[color=red,opacity=0.5] (-1,0) circle (2pt);
\end{tikzpicture}
}

\newcommand{\pathsetting}{
    Assume the setting of \Cref{sing_basics:ex_path}: let $\Psi :[0,1]  \rightarrow \C^d$ be a continuous path and fractionally continuable at $t_0 \in (0,1)$. Denote by $\Theta: B_\epsilon(t_0) \rightarrow \C^d \Smz$ the fractional continuation of $\Psi$ at $t_0$, $\lambda: B_\epsilon(t_0) \rightarrow \C$ the removable prefactor, $A \subset \{1, \ldots, d\}$ the set of all indices such that $\Psi$ is not constantly zero on a neighborhood of $t_0$ and $i \in A$ the selected component as defined in \Cref{sing_basics:prefac}.
}
\newcommand{\pathsettingwoprefactor}{
Assume the setting of \Cref{sing_basics:ex_path}: let $\Psi :[0,1]  \rightarrow \C^d$ be a continuous path and fractionally continuable at $t_0 \in (0,1)$. Denote by $\Theta: B_\epsilon(t_0) \rightarrow \C^d \Smz$ the fractional continuation of $\Psi$ at $t_0$, $A \subset \{1, \ldots, d\}$ the set of all indices such that $\Psi$ is not constantly zero on a neighborhood of $t_0$. and $i \in A$ the selected component as defined in \Cref{sing_basics:prefac}.
}
\newcommand{\pathsettingwocomp}{
    Assume the setting of \Cref{sing_basics:ex_path}: let $\Psi :[0,1]  \rightarrow \C^d$ be a continuous path and fractionally continuable at $t_0 \in (0,1)$. Denote by $\Theta: B_\epsilon(t_0) \rightarrow \C^d \Smz$ the fractional continuation of $\Psi$ at $t_0$, $\lambda: B_\epsilon(t_0) \rightarrow \C$ the removable prefactor, $A \subset \{1, \ldots, d\}$ the set of all indices such that $\Psi$ is not constantly zero on a neighborhood of $t_0$.
}

\crefname{algocf}{alg.}{algs.}
\Crefname{algocf}{Algorithm}{Algorithms}

\begin{document}

\begin{frontmatter}

\title{Non-standard Analysis in Dynamic Geometry}

\author{Michael Strobel}
\address{Technical University of Munich, Germany}
\ead{strobel@ma.tum.de}
\ead[url]{https://www-m10.ma.tum.de/bin/view/Lehrstuhl/MichaelStrobel}

\begin{abstract}
    We will present the benefits of using methods of non-standard analysis in dynamic projective geometry. One major application will be the desingulariazation of geometric constructions.
\end{abstract}

\begin{keyword}
dynamic projective geometry, non-standard analysis, removal of singularities
\end{keyword}
\end{frontmatter}

\section{Geometry and Infinity}

In this article we will integrate concepts of non-standard analysis (NSA) into dynamic projective geometry. One major application of the developed theory will be the automatic removal of singularities in geometric constructions.

Non-standard analysis allows the precise handling of infinitely small (infinitesimal) and infinitely large (unlimited) numbers which admits statements like: ``there is an $\epsilon > 0$ such that $\epsilon < r \; \forall r \in \R^+$'' or ``there is an $H$ such that $H > r \; \forall r \in \R^+$''. This obviously violates the Archimedean axiom that for every real number there is always a larger natural number. Therefore such fields that contain infinitesimal and unlimited numbers are also called ``non-Archimedean''. To embed this theory into the ``standard'' theory, one may use field extensions of $\R$ and $\C$. We will call them $\Rs$ (hyperreal numbers) and $\Cs$ (hypercomplex numbers), and one may seamlessly integrate these numbers into well-established (we will refer to it by the adjective ``standard'') analysis in the sense of Weierstrass. Although the construction of these fields is highly non-trivial, its mere application is very intuitive.

One strength of projective geometry is the natural and comprehensive integration of infinity. For example: consider the intersection of two (disjoint) lines. Euclidean geometry has to distinguish two cases: the lines intersect or they do not intersect (if they are parallel). In projective geometry there is \textit{always} a point of intersection, it just may be infinitely far way. 

In practice, projective geometry is usually carried out over $\R$ or $\C$, but its not restricted to these canonical fields. One may consider arbitrary fields for geometric operations and this will enable us to combine the mathematical branches of projective geometry and non-standard analysis.

Using infinitesimal elements in geometry is not a new concept, it is actually one of the oldest concepts. Already Leibniz and Newton used them for geometric reasoning.

Dynamic (projective) geometry allows us to describe the movement of geometric constructions and implement them on a computer. In dynamic constructions, singularities naturally arise since these may describe particularly interesting configurations. We will look into methods for resolving singularities and situations in which standard theory does not offer concise and consistent solutions. 
This will naturally lead to the integration of non-standard analysis into projective geometry. Using this extension of projective geometry, we will be able to develop the theory of proper desingulariazation. After expanding this theory we will also develop algorithms and a real time capable implementation to integrate the theory into a dynamic geometry system.

We will illustrate the theory with several basic examples, like circle-circle intersections or point-line configurations. The scope of the article is not to show advanced constructions, but rather to focus on geometric primitive operations. 
Advanced configurations are usually built up by primitive operations, hence if every construction step can be desingularized then the whole constructions will be non-singular. We hope that the reader already recognizes the capabilities of the presented methods with elementary examples.  

\section{Singularities of Geometric Constructions}
\label{chap:sing-geo}
Building up geometric constructions using primitive operations, for instance the connecting line of two points or the intersection of two circles, relies on certain non-degeneracy assumptions like the two points or circles being distinct. 
Even if primitive operations are implemented in a way that does not 
introduce artificial degenerate situations, combination of those can produce such situations.
Consider a geometric construction
that should create a certain element depending on some free elements. As an example consider the calculation of an orthogonal bisector of two points $A$ and $B$. In a dynamic geometry system a user may implement such a construction 
by applying a sequence of geometric primitive operations. One  possible (perhaps not the most clever) way to do this is by drawing two circles of radius $1$ around the points, intersecting them and then joining the two intersections. How should such a construction behave when the user moves the original points?
See \Cref{sing_basics:fig:ortho_bisec} for a visualization of the situation.

It is clear that the two real points of intersection and their connecting line should be shown when $A$ and $B$ are at a distance less then two. At a distance greater than two one could argue (and this is a modeling step)
what the desired behavior should be. We propose the following behavior: the intersection points have become complex. Joining them creates a line with complex homogeneous coordinates. However a common factor can be 
extracted from this complex coordinate vector and the line can be interpreted as a real line: The orthogonal bisector of $A$ and $B$. But what should happen if $A$ and $B$ are exactly in a situation 
where their distance equals $2$? In this case the two circle intersections coincide and we do not have a well defined connecting line. However, in a sense the situation behaves like a removable singularity. We can consider the situation as a limit process: in an epsilon neighborhood of the singularity the position of the line stably converges to the same situation. The question is how to automatically detect such cases and automatically \textit{desingularize} them. 
\begin{figure}
    \centering
    \resizebox{\textwidth}{!}{%
    \begin{tikzpicture}
    \draw (0,0) circle (1cm);
    \fill[] (0.0,0.0) circle (1pt);
    \fill[] (1.0,0.0) circle (1pt);
    \draw (1,0) circle (1cm);
    \draw (-1.8,0) -- (2.2,0);
    \draw[color=blue!80!black, line width = 0.5pt] (0.5,-2) -- (0.5,2);
    
    \fill[color=red] (0.5,0.866) circle (2pt);
    \fill[color=green] (0.5,-0.866) circle (2pt);
    \end{tikzpicture}
    \begin{tikzpicture}
    \draw (-1,0) circle (1cm);
    \fill[] (-1.0,0.0) circle (1pt);
    \draw (1,0) circle (1cm);
    \fill[] (1.0,0.0) circle (1pt);
    \draw (-2.2,0) -- (2.2,0);
    \draw[color=blue!80!black,dashed,  line width = 0.5pt] (0,-2) -- (0,2);
    
    \fill[color=green] (0,0) circle (2pt); 
    \fill[color=red] (0,0) -- (180:2pt) arc (180:0:2pt) -- cycle; 
    \end{tikzpicture}
    \begin{tikzpicture}
    \draw (0,0) circle (1cm);
    \fill[] (0.0,0.0) circle (1pt);
    \draw (2.5,0) circle (1cm);
    \fill[] (2.5,0.0) circle (1pt);
    \draw (-1.2,0) -- (3.8,0);
    \draw[color=blue!80!black, line width = 0.5pt] (1.25,-2) -- (1.25,2);
    \end{tikzpicture}
    }%
    \caption{Construction of an orthogonal bisector with removable singularity (middle).}
    \label{sing_basics:fig:ortho_bisec}
\end{figure}
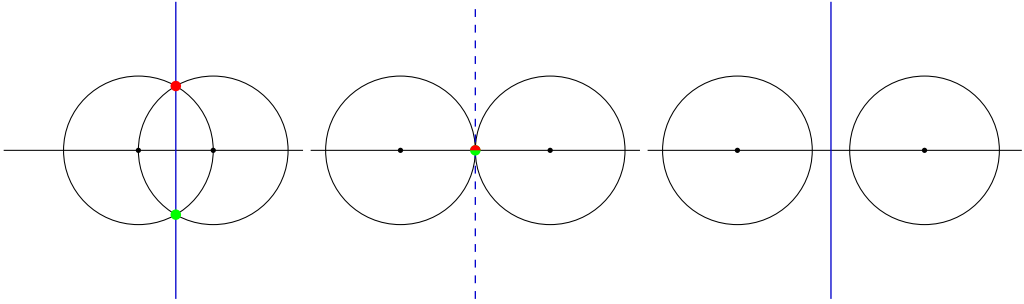
\subsection{Modelling Dynamic Geometry}
First, we have to define our model for the removal of singularities. \jrguk have already laid the foundations in their article ``\textsc{Grundlagen Dynamischer Geometrie}'' \cite{kortenkamp2001grundlagen}. 
In their article, Euclidean representations are used for geometric objects to obtain a classical notion of continuity. Here, we will define continuity in a topological sense and preserve the benefits of a projective space at the expense of the introduction of topological machinery.

A well known topological characteristic of topological continuity is the following:
\begin{lemma}[Topological Continuity]
    \label{sing_basics:conti_inout}
    Let $X,Y$ be two topological spaces and $X_{/ \sim}$ a quotient space of $X$ by an equivalence relation $\sim$ on $X$ with canonical quotient map $\pi: X \rightarrow X_{/ \sim}, x \mapsto [x]$. Then it holds true:
    \begin{enumerate}
        \item A function $f: X_{/ \sim} \rightarrow Y$ is continuous, if and only if $f \circ \pi: X \rightarrow Y$ is continuous.
        \item Let $\phi: Y \rightarrow X_{/ \sim}$ be a function. If there is a continuous function $\Phi: Y \rightarrow X$ with $\phi = \pi \circ \Phi$, then $\phi$ is continuous. 
    \end{enumerate} 
\end{lemma}
%
\begin{remark}
    Especially the second point is important in our situation for us: if we can find a continuous function $f: \R \rightarrow \C^3 \setminus \{0\}$, then its counterpart $[f]$ in $\CP^2$ is automatically continuous. We will exploit this fact very often, since this induces a simple recipe to find continuous extension, we will call them $C^0$--continuations, of a function.
\end{remark}
\begin{remark}
Let $X$ be a topological space and $X_{/ \sim}$ be the quotient space defined by an equivalence relation $\sim$ on $X$. By abuse of notation, we will write $x= y$ for all $[x], [y] \in X_{/ \sim}$ if we mean $[x] = [y]$.
\end{remark}
The term ``singularity'' is used to describe a point of a function where the function is undefined or not well-behaved. We will encounter singularities in at least two variations: for functions $f: [0,1] \rightarrow \C^{d+1}$ for which the preimage of the zero vector is not the empty set (we will motivate this after the definition) and the more classical variant from complex analysis: if a fraction of two functions attains $0$ for nominator and the denominator at the same time. We will see that these notions are closely related for our purposes.

\begin{definition}[Singularity]
    Let $f: [0,1] \rightarrow \C^{d+1}$ be a continuous function and let $t_0 \in [0,1]$. We call $f$ \textit{singular at $\mathit{t_0}$} if $f(t_0) = 0$. And we call $t_0$ a \textit{singularity}.
\end{definition}
\begin{remark}
    This definition is motivated by the fact that the projection of the zero vector is not an element of the complex projective space $\CP^d$. While a function $f: [0,1] \rightarrow \C^{d+1}$ is well defined if $f(t_0) = 0$ the same function seen as a function in the projective space $[f]: [0,1] \rightarrow \CP^d, t \mapsto [f(t)]$ is not a well-defined function for $t_0$.
\end{remark}
\begin{definition}[Isolated Point]
  Let $I \subset [0,1]$ be a subset of the unit interval. We call a point $x \in I$ \textit{isolated} if $\not \exists \epsilon > 0: B_\epsilon(x) \subset I$.
\end{definition}
As mentioned before we want do develop techniques to remove such singularities if possible.
Therefore, we will now define a $C^0$-continuation which is a notion of a continuous extension of a function.
\begin{definition}[$C^0$-continutation]
    \label{sing_basics:topo_c0_conti}
    Let $I \subset [0,1]$ be a subset of the unit interval with no isolated points, $s_1, \ldots, s_n \in I$ and let $I':= I \setminus \{s_1, \ldots, s_n\}$. Let $\phi:  I' \rightarrow \CP^d$ be continuous. If there is a continuous function $\hat \phi: \hat I \rightarrow \CP^d$ with $I' \subsetneq \hat I \subset I$ and $\phi(t) = \hat \phi(t) \, \forall t \in I'$
    (note that ``$=$'' here means the same equivalence class), 
    then we call $\phi$ \textit{$\mathit{C^0}$--continuable on $\mathit{\hat I}$}
   and 
   $\hat \phi$ the \textit{$\mathit{C^0}$--continuation (or \textit{desingulariazation}) of $\mathit{\phi}$ on $\mathit{\hat I}$}. And all $s \in \hat I \setminus I'$ are called \textit{removable singularities}.
\end{definition}

%
\begin{remark}
   In other words: in the previous definition we consider a function $\phi$ which is singular at $t_0$. If we can find another function $\hat \phi$ which coincides with $\phi$ on the domain of $\phi$ but is also continuous at $t_0$, then we can continuously extend the function $\phi$ at $t_0$, with the value of $\hat \phi$ at $t_0$, and remove the singularity.
\end{remark}
\begin{definition}[Quasi Continuous]
\label{sing_basics:topo_quasi_conti}
Let $I \subset [0,1]$ be a subset of the unit interval $[0,1]$ with no isolated points and $I' := I \setminus \{ s_1, \ldots, s_n\}$ for $s_1, \ldots, s_n \in [0,1]$ ($n\in \N$). We will define quasi continuity for functions which map from an interval to the set of points $\PC$, lines $\LC$ and conics $\CC$ in $\CP^2$ respectively:
    \begin{itemize}
        \item $\phi: I' \rightarrow \PC$ is called \textit{quasi continuous}, if there exists a continuous function $\psi: I \rightarrow \PC$ such that $\phi(t) = \psi(t)$ for all $t \in I'$. 
        \item $\phi: I' \rightarrow \LC$ is called \textit{quasi continuous}, if there exists a continuous function $\psi: I \rightarrow \LC$ such that $\{ p \in \PC \bbar \langle p, \phi(t) \rangle = 0 \} = \{ p \in \PC \bbar \langle p, \psi(t) \rangle = 0 \}$ for all $t \in I'$.
        \item $\phi: I' \rightarrow \CC$ is called \textit{quasi continuous}, if there exists a continuous function $\psi: I \rightarrow \CC$ such that $\{ p \in \PC \bbar p^T A_{\phi(t)} p = 0 \} = \{ p \in \PC \bbar p^T A_{\psi(t)} p = 0 \}$ where $A_{\phi(t)}$ and $A_{\psi(t)}$ denote the associated matrices of $\phi(t)$ and $\psi(t)$  for all $t \in I'$.
    \end{itemize}
In each case we will call the function $\psi$ the \textit{quasi continuation}.
\end{definition}
\begin{remark}
    In essence, a quasi continuous function can be extended continuously on a larger domain. Furthermore quasi continuity here is a generalized version of the quasi continuity of \cite{kortenkamp2001grundlagen}.
\end{remark}
\begin{lemma}[Quasi Continuity and $C^0$-Continuations] 
Let $I \subset [0,1]$ be a subset of the unit interval $[0,1]$ with no isolated points and $I' := I \setminus \{ s_1, \ldots, s_n\}$ for $s_1, \ldots, s_n \in [0,1]$ ($n\in \N$). Let $\phi: I' \rightarrow \CP^d$ be a continuous function. $\phi(t)$ is quasi continuous if and only if for all $s \in I \setminus I'$ there is a $C^0$--continuation. 
\end{lemma}
\begin{proof}
%

    If $\phi$ is quasi continuous then the claim is clear just by definition.

    Now let there be a $C^0$-continuation for all $s_i \in I \setminus I'$.
By the prerequisites, we know that for every $s_i$ there is a $C^0$--continuation $ \psi_i: I' \cup \{s_i\} \rightarrow \CP^d$ for $i \in \{1, \ldots, n\}$.
 
We can now ``glue'' the $C^0$--continuations $\psi_i$ to one continuous function on $I$: define $\psi(t): I \rightarrow \CP^d$ by
\[
    \psi(t):=
    \begin{cases}
        \phi(t),& \text{ if } t \in I'\\
         \psi_i(t), & \text{ if } t = s_i \text{ for } i \in \{1, \ldots, n \}
    \end{cases}
\]
By construction $\psi(t)$ coincides with $\phi(t)$ for all $t \in I'$. Furthermore by assumption $\psi(t)$ is continuous for all $t \in I$ since a $C^0$--continuation exists for all $t \in I \setminus I'$ and is given by $\psi_i$. But this means that $\phi$ is quasi continuous.
\end{proof}
\begin{example}[$C^0$-Continutation and removable Singularities]
    Let $I = [0,1]$ and $I':= I \setminus \{ \frac{1}{2} \}$. Define $\phi(t): I' \rightarrow \CP^2$ by 
    \[
        \phi(t):= \left[
            \begin{pmatrix}
                \sin\left(t- \frac{1}{2}\right) \\ t - \frac{1}{2} \\ t - \frac{1}{2}
        \end{pmatrix} 
    \right]
    \]
    First note that $\phi(t)$ is continuous on $I'$ by \Cref{sing_basics:conti_inout}. 
    
    Now define $\psi: I \rightarrow \C^3 \Smz$ with 
    \[
        \psi(t):= \begin{pmatrix}
            \frac{\sin\left(t- \frac{1}{2}\right)}{t - \frac{1}{2}} \\ 1 \\ 1
    \end{pmatrix} \text{ if } t \neq \frac{1}{2} \text{ and } \psi \left(\frac{1}{2}\right) := \begin{pmatrix}
        1 \\ 1 \\ 1
    \end{pmatrix}
    \]
    Obviously, $\psi$ is continuous on $I'$, but is it also continuous on $I$ by de L'Hospital's rule. Now define $\hat \phi := \pi \circ \psi = [\psi]$, then $\phi(t) = \hat \phi(t) \, \forall t \in I'$, which can easily be seen if one divides $\phi$ by its $z$--coordinate. Hence $\hat \phi$ is a topological $C^0$--continuation and $t = \frac{1}{2}$ is a removable singularity. Furthermore, this also shows that $\phi$ is quasi continuous: the $C^0$-continuation here coincides with the quasi continuation since there is only one singularity to resolve. 
\end{example}
The next theorem will give us a way to resolve singularities: given a function $f: [0,1] \rightarrow \C^d$ which attains $0$, and therefore has a singularity, we might find a $C^0$-continuation of $[f]$ by dividing the whole image by one of the components of the function $f$.

\begin{theorem}[Existence of a Continuous Path, \cite{kortenkamp2001grundlagen}]
    \label{sing_basics:ex_path}
    Let $\Psi :[0,1]  \rightarrow \C^d$ be a continuous path and $t_0 \in [0,1]$. Let $A\subset \{1, \ldots, d\}$ be the set of indices of $\Psi$--components which are not constantly $0$ on a neighborhood of $t_0$. If 
    \begin{enumerate}
        \item $A \neq \emptyset$ 
        \item for all $i,j \in A: \, \Psi_i(t)/\Psi_j(t)$ or $\Psi_j(t)/\Psi_i(t)$ has a removable singularity or is continuous at $t_0$,
    \end{enumerate}
    then there are an $\epsilon >0$, a continuous path $\Theta: B_\epsilon(t_0) \rightarrow \C^d \Smz$ and a function $\lambda: B_\epsilon(t_0) \rightarrow \C$ such that $\lambda(t) \cdot \Theta(t) = \Psi(t)$ for all $t\in B_\epsilon(t_0) \setminus \{t_0\}$.

The described functions $\Theta, \lambda$  have the form 
    \[
        \Theta(t) = \frac{\Psi(t)}{\Psi_i(t)}, \quad \lambda(t) = \Psi_i(t)
    \]
    where $i \in A$ is appropriately chosen such that all coordinate entries of $\Theta$ have a removable singularity or are continuous.
\end{theorem}
\begin{remark}
    The last theorem might seem a bit unimposing at first look, but it is actually a criterion for the situation when a singularity is removable! If the prerequisites are fulfilled, we can split $\Psi$ into two parts: a vanishing ($\lambda$) and a non-vanishing ($\Theta$) part at $t_0$.
    Since $\Theta: B_\epsilon(t_0) \rightarrow \C^d \Smz$ is non--vanishing at $t_0$ we can use $[\Theta]$ as a $C^0$-continuation of $[\Psi]$ at $t_0$. We will make note of this fact in \Cref{sing_basics:frac_conti_lemma}. 
\end{remark}
First we will establish more notation which is implied by the previous theorem.
\begin{definition}[Fractionally Continuable]
    \label{sing_basics:prefac} 
    If a continuous function $\Psi :[0,1]  \rightarrow \C^d$ fulfills the requirements of the \Cref{sing_basics:ex_path} at $t_0 \in [0,1]$, we call $\Psi$ \textit{fractionally continuable at $\mathit{t_0}$}. And we call the fractional function $\Theta: B_\epsilon(t_0) \rightarrow \C^d \Smz$ as defined in \Cref{sing_basics:ex_path} the \textit{fractional continuation of $\mathit{\Psi \text{ around } t_0}$}. We call the function $\lambda: B_\epsilon(t_0) \rightarrow \C$ as defined in \Cref{sing_basics:ex_path} \textit{removable prefactor} or, more shortly, \textit{prefactor}. We call the component $i \in A$, which defines $\lambda$ in \Cref{sing_basics:ex_path}, the \textit{selected component at $\mathit{t_0}$}.


If additionally $\lambda$ is analytic, we call $\lambda$ an \textit{analytic removable prefactor} or shortly \textit{analytic prefactor}. 

If the function $\Theta$ is analytic on $B_\epsilon(t_0)$, then we call the function $\Psi$ \textit{analytical fractionally continuable at $\mathit{t_0}$}. Furthermore we call $\Theta$ the \textit{analytic fractional continuation of $\mathit{\Psi \text{ around } t_0}$}.
\end{definition}
The next lemma shows that if a function is fractionally continuable then it is also $C^0$-continuable.
\begin{lemma}[Fractionally Continuable $\Rightarrow C^0$-continuable]
    \label{sing_basics:frac_conti_lemma}
Let $I \subset [0,1]$ be a subset of the unit interval with no isolated points,
 $\Psi : I  \rightarrow \C^{d+1}$ be continuous and the preimage of zero, denoted by $S:= \Psi^{-1}(0)$, be discrete (\ie $S = \{s_1, \ldots, s_n\}$ for an $n \in \N$) and define $I':= I \setminus S$. Let $\Psi$ be fractionally continuable for all $s_i \in  \hat I \cap S$ where $I' \subsetneq \hat I \subset I$.
    Then $[\Psi]: I' \rightarrow \CP^d \text{ with } t \mapsto [\Psi(t)]$ is $C^0$--continuable on $\hat I$. And if it furthermore holds true that $\hat I = I$, then $[\Psi]$ is quasi continuous.
\end{lemma}
\begin{proof}
    Let $i \in \{1,\ldots,n\} $. Since $\Psi$ is fractionally continuable for all $s_i \in \hat I \cap S$, by \Cref{sing_basics:ex_path} there exist $\epsilon_i >0$ and a continuous functions $\Theta_i: B_{\epsilon_i}(s_i) \rightarrow \C^{d+1} \Smz$ and $\lambda_i: B_{\epsilon_i}(s_i) \rightarrow \C$ such that 
    	\[
      [\lambda_i(t) \cdot \Theta_i(t)] = [\Psi(t)] 
      \text{ if and only if } [\Theta_i(t)] = [\Psi(t)]  \text{ for all } t \in  B_{\epsilon_i}(s_i) \setminus \{s_i\}.
        \] 
        As $\Theta_i$ is continuous and non-vanishing on whole $B_{\epsilon_i}(s_i)$, this means that $\Psi$ has a $C^0$-continuation for all $s_i \in \hat I \cap S$. Finally, by definition the property that $\hat I = I$ is equivalent for $[\Psi]$ being quasi continuous.
\end{proof}
In their proof of \Cref{sing_basics:ex_path} Kortenkamp and Richter-Gebert assume that possibly occurring singularities may be removed, although they do not state how this desingulariazation should be achieved. From a mathematical point of view the problem is quite easy: either one checks for the classical epsilon-delta criterion or equivalently uses a limit process. 

From a computational point of view, the problem is completely different. D.\ Richardson proved in \cite{richardson1969solution} that zero testing for certain classes of functions is undecidable. As D.~Gruntz points out in his Ph.D.\ Thesis (\cite{gruntz1996computing}), this problem can be reformulated in order to check for the continuity of a function and therefore the problem is also undecidable (for certain functions). This result is very devastating at first sight, but the situation is not that bad if one assumes certain regularity of functions. We are in the lucky position that the functions we are considering are even analytic, which is a very strong property. Nevertheless, we will face problems at essential singularities, as we will see later on.

\subsection{The Analytic Case}
\label{sing_basics:sec:analytic_case}
The problem of handling removable singularities in geometric constructions is closely related to the problem of removable singularities in complex analysis since one can analyze quotients of functions, as we saw in the theorem by \jrguk (\Cref{sing_basics:ex_path}). For quotients of holomorphic functions, this is classic theory and can be found in every textbook on complex variables (see \eg \cite{ krantz2012handbook}). The standard procedure is to extend the function in a continuous way and this is equivalent to analytic continuation by Riemann's theorem (\cite{krantz2012handbook} \page 42-43). More concretely, one usually expands the functions to its unique power series and cancels common zeros at the singularity.

Firstly, we will analyze the connection of roots and singularities. If we speak of roots we usually mean a point where a function attains $0$. We will refer to functions like $\sqrt[n]{z}$ as ``radical expressions'' or more shortly ``radicals'' to avoid confusion. 

We will exploit the following fact well known from complex analysis:
\begin{lemma}[{Roots and Singularities}]
    \label{sing_basics:zersing}
    Let $D \subset \C$ be a domain and $z_0 \in D$. Let $f,g: D \rightarrow \C$ be analytic. If both $f$ posses a root of order $k \in \N$ at $z_0$ and $g$ a root of order $l \in \N$ at $z_0$, $k \geq l$ then $h(z): = f(z)/g(z)$ has a removable singularity at $z_0$. Furthermore $h$ can be written by 
    \[
        h(z) = \frac{(z-z_0)^k \cdot \hat f(z)}{(z-z_0)^l \cdot \hat g(z)}
    \]
    where both $\hat f, \hat g$ are analytic and $\hat g$ does not vanish at $z_0$.
\end{lemma}

We will now apply the previous insights to our setting.
\begin{lemma}[Component Decomposition]
	\label{sing_basics:comp_decomp}
        Let $\Psi: I \subset (0,1) \rightarrow \C^d$ be analytic in every component and analytically fractionally continuable at $t_0 \in I$.
        Then we can decompose the components of $\Psi(t)$ around $t_0$ into non-vanishing and (possibly) vanishing parts: there exist $\epsilon > 0$ and $k \in \N$ such that:
            \[
                \Psi(t) = (t-t_0)^k \cdot \tilde \Psi(t)\text{ for all } t \in B_\epsilon(t_0)
    \]
    where $\tilde \Psi(t): B_\epsilon(t_0) \rightarrow \C^d \Smz$ is an analytic function. 
    The function $[\Psi] \big|_{B_\epsilon(t_0)}: B_\epsilon(t_0) \rightarrow \CP^{d-1}, t \mapsto [\Psi(t)]$ is $C^0$--continuable on 
   at $t_0$ with $C^0$-continuation $[\tilde \Psi]$.
\end{lemma}
\begin{remark}
    Note that in the previous lemma $k$ is finite but might also be zero. This describes the case when the function $\Psi$ has at least one component which does not vanish on $B_\epsilon(t_0)$. 
\end{remark}
\begin{proof}
The first part of the claim is obvious by \Cref{sing_basics:zersing}. We will only proof the $C^0$-continuation statement: since it holds true that 
            \[
                \Psi(t) = (t-t_0)^k \cdot \tilde \Psi(t)\text{ for all } t \in B_\epsilon(t_0)
    \]
    and we know that $\tilde \Psi$ and $(t-t_0)^k$ are non-vanishing on $ B_\epsilon(t_0) \setminus \{t_0\}$, we find:

            \[
                [\Psi(t)] = [(t-t_0)^k \cdot \tilde \Psi(t)] = [\tilde \Psi(t)] \text{ for all } t \in B_\epsilon(t_0) \setminus \{t_0\}
    \]
    So the equivalence classes of $[\Psi (t)]$ and $[\tilde \Psi (t)]$ coincide everywhere on $B_\epsilon(t_0) \setminus \{t_0\}$. Since $[\tilde \Psi (t)]$ is continuous (even analytic) and non-vanishing on whole $B_\epsilon(t_0)$ we conclude that $[\tilde \Psi (t)]$ is a $C^0$-continuation of $[\Psi]$ at $t_0$.
\end{proof}
\begin{remark}
   More figuratively speaking, we factor out the minimum amount of roots at the singularity $t_0$ such that not all components vanish.
\end{remark}
\begin{definition}[Assoticated Polynomial and Order of Removable Singularity]
	\label{sing_basics:ass_poly}
        The function $(t-t_0)^k$ ($k  \in \N )$ in \Cref{sing_basics:comp_decomp} is called \textit{$\mathit{t_0}$--associated polynomial of $\Psi$} or shortly \textit{associated polynomial}, if it is obvious at which point we consider the series expansion. Usually we will denote the associated polynomial by $p(t)$. We call $k$ the \textit{order of the removable singularity}. 
\end{definition}

%
%
\begin{remark}
    As we saw in the two preceding statements the associated polynomial a removable prefactor which ``pushes'' functions unnecessarily to zero. So removing the associated polynomial removes also the singularity.
\end{remark}
    The assumption that the functions have to be analytic is not far fetched. Although we want to model continuous movement, a much weaker premise, Kortenkamp and Richter-Gebert show that problems in dynamic geometry can be formulated as algebraic systems (\cite{kortenkamp2001grundlagen}). Employing classic theory of algebraic curves as described in \cite{brieskorn2012plane}, one can describe these algebraic curves locally as generalized power series, so called Puiseux series which are complex analytic functions. 

\section{Singularities and Derivatives}
\label{sec:sing_and_deriv}
Now we want to investigate the connection between analytic prefactors and derivatives. We will find that out removable singularities that do not involve radical expression can be easily resolved using derivatives. 
\begin{lemma}[Removal of Analytic Prefactors]
	\label{sing_basics:removal_ana}
        Let $\Psi: I \subset (0,1) \rightarrow \C^d$ be analytic in every component. Futhermore, let $\Psi$ be analytic fractionally continuable at $t_0 \in I$, with a root of order $k \in \N$ at $t_0$ in component $i$, where $i$ the selected component at $t_0$.
        Then an analytic fractional continuation $\Theta : B_\epsilon(t_0) \rightarrow \C^d $ on an open subset around $t_0$ $\Theta$ is given by:
    \[
        \Theta_j(t):= 
        \begin{cases}
            \Psi_j(t) / \Psi_i(t), &\text{ if } t\neq t_0 \\
            \Psi_j^{(k)}(t) / \Psi_i^{(k)}(t) , &\text{ if } t = t_0
        \end{cases}
    \]
    for all $j \in \{1, \ldots, d\}$.
\end{lemma}
\begin{proof}
  Direct application of the complex version of de L'Hospital's rule (\cite{freitag2000funktionen}).
\end{proof}
The last Lemma motivates a more direct way to the remove singularities. Instead of evaluating the derivatives of quotients of components one can also apply the derivatives directly to the components, which will be our next statement.
\begin{theorem}[Direct Derivation]
    \label{sing:kth_deriv}
 Let $\Psi: I \subset (0,1) \rightarrow \C^d$ be analytic in every component and analytical fractionally continuable at $t_0 \in I$,
 with $p(t) = (t-t_0)^k$ an analytic prefactor of order $k\in \N$.
    Then the $k$'th derivative of $\Psi$ at $t_0$ is a non--zero scalar multiple of the analytic fractional continuation $\Theta$ at $t_0$: there is $\tau \in \C \setminus \{0\}$:
    \[
	    \Theta(t_0) = \tau \cdot \Psi^{(k)}(t_0)
    \]

    Furthermore, the continuous function $[\Psi]: I \setminus \{\Psi^{-1}(0)\} \rightarrow \CP^{d-1}, t \mapsto [\Psi(t)]$ has a $C^0$--continuation at $t_0$.
\end{theorem}
\begin{proof}
    The proof is essentially based on the product rule and rules of derivation for polynomials. By \Cref{sing_basics:comp_decomp} we know that we can find a decomposition of $\Psi$ such that $\Psi(t) = p(t) \cdot \tilde \Psi(t)=  (t-t_0)^k \cdot \tilde \Psi(t)$ with $\tilde \Psi(t)$ analytic and non-vanishing at $t_0$, then
    \begin{align*}
	    \Psi(t_0) &=  p(t_0) \cdot \tilde \Psi(t_0) \\
	    \Rightarrow \Psi^{(1)}(t_0) &= p^{(1)}(t_0) \cdot \tilde \Psi(t_0) + \underbrace{p(t_0)}_{= 0 \text{ by definition}} \cdot \Psi^{(1)}(t_0)\\
	    \Rightarrow \Psi^{(1)}(t_0) &= p^{(1)}(t_0) \cdot \tilde \Psi(t_0) \\
         \vdots \\ 
	\Rightarrow \Psi^{(k)}(t_0) &= p^{(k)}(t_0) \cdot \tilde \Psi(t_0) = k! \cdot \tilde \Psi(t_0) 
\end{align*}
On the other, hand we know that the $j$'th component of the analytic fractional continuation $\Theta$ is given by  
\[
	\Theta_j(t_0)= 	\Psi_j^{(k)}(t_0) / \Psi_i^{(k)}(t_0),
\]
where $i \in \{1, \ldots, d\}$ is the selected component at $t_0$, which we proved in \Cref{sing_basics:removal_ana}.
Using \Cref{sing_basics:comp_decomp} we can write this as: 
\begin{align*}
\Theta_j(t_0)&= 	\Psi_j^{(k)}(t_0) / \Psi_i^{(k)}(t_0)\\ &= \left. \left(\frac{\dd^k}{\dt^k} \, (p(t)  \cdot \tilde \Psi_j(t))\right)\right\rvert_{t=t_0} / \left.\left(\frac{\dd^k}{\dt^k} \, (p(t) \cdot \tilde \Psi_i(t))\right) \right\rvert_{t=t_0}\\
				  &= \left(k! \cdot \tilde \Psi_j(t_0)\right) / \left(k! \cdot \tilde \Psi_i(t_0)\right)\\
	     &= \tilde \Psi_j(t_0)) / \tilde \Psi_i(t_0)
\end{align*}
So we can set, $\tau= \frac{k!}{\tilde \Psi_i(t_0)}$ where $\tilde \Psi_i(t_0) \neq 0$ by construction, and have:
\[
	\Theta(t_0) = \tau \cdot \Psi^{(k)}(t_0).
\]
Finally, this also means that 
\[
    [\Theta(t_0)] = [\tau \cdot \Psi^{(k)}(t_0)] = [\Psi^{(k)}(t_0)]
\]
which shows the $C^0$-continuation property.
\end{proof}
Now we will give a series of examples which illustrate the concept discussed. 
\begin{remark}
    We defined quasi continuity only for functions whose pre-image is a subset of the unit interval. We will deviate from this assumption from time to time if examples can be easier written down using different (bounded) intervals. Of course this is without loss of generality, one can simple re-parameterize the functions at any time to fit into a subset of the unit interval.
\end{remark}
\begin{example}[Farpoint of Parallel Lines]
	Let $a = (0,1,0)^T$ and $b(t):= (0,1,(t-\tfrac{1}{2})^3), \; t \in [0,1]$ be families of two lines
    in $\RP^2$. Denote their point of intersection as $P:= a \wedge b$ with
    \[
        P(t) =  \begin{pmatrix}  (t-\tfrac{1}{2})^3 \\ 0 \\ 0 \end{pmatrix} =(t-\tfrac{1}{2})^3 \begin{pmatrix}  1 \\ 0 \\ 0 \end{pmatrix}.
    \]
    So it is easy to see that their point of intersection is either a far point or not defined (for $t = \tfrac 1 2$). The singularity at $t=\tfrac 1 2$ has order 3 and we can apply the developed theory.
    By \Cref{sing:kth_deriv} we have to evaluate three derivatives
    of $P$: 
    \begin{align*}
        P'(t) =   3\cdot(t-\tfrac{1}{2})^2  \begin{pmatrix}  1 \\ 0 \\ 0 \end{pmatrix}, \; 
        P''(t) =  6\cdot(t-\tfrac{1}{2}) \begin{pmatrix}  1 \\ 0 \\ 0 \end{pmatrix}, \; 
        P'''(t) = 6 \cdot \begin{pmatrix}  1 \\ 0 \\ 0 \end{pmatrix}\\
    \end{align*}
    Then the function $[P(t)]: [0,1] \setminus \{\tfrac{1}{2}\} \rightarrow \RP^2$ is quasi continuous with quasi continuation $[\hat P(t)]: [0,1] \rightarrow \RP^2$
    \[
    [\hat P(t)] :=  
\begin{rcases}
    \begin{dcases}
[P(t)], &\text{ if } t \neq \tfrac{1}{2},\\
        [ (1,0,0)^T], &\text{ if } t = \tfrac{1}{2}
    \end{dcases}
  \end{rcases}
 = \begin{pmatrix} 1 \\ 0\\ 0 \end{pmatrix} \; \forall t \in [0,1].
    \]
    One can find an illustration in \Cref{sing_basics:para_line_fig}.
    \begin{figure}
      \centering 
\resizebox{0.9\linewidth}{!}{
    \begin{tikzpicture}
    \draw[thick] (-0.3,0) -- (11,0);
    
    \coordinate (inf) at (10,0);
    \fill (inf) circle [radius=2pt];
    \node at (inf) [below = 2mm] {$(1,0,0)^T$};
    \node at (inf) [below = 8mm] {$P(t) = a \wedge b = ((t-\frac 1 2)^3, 0,0)^T$};
    
     \coordinate (t1) at (9,2.5);
     \node at (t1) [above = 0mm] {${t=0}$};
     \draw [dashed,shorten <=0.1cm, shorten >=0.1cm] (t1) -- (7,2);
    
     \coordinate (t2) at (10.5,1.5);
     \node at (t2) [above = 0mm] {${t=\frac 1 4}$};
     \draw [dashed,shorten <=0.1cm, shorten >=0.1cm] (t2) -- (8,1);
    
     \coordinate (t3) at (12,0.5);
     \node at (t3) [above = 0mm] {\color{red}${t=\frac 1 2}$};
     \draw [dashed,shorten <=0.1cm, shorten >=0.1cm] (t3) -- (8,0);
    
     \coordinate (a) at (1.5,-0.75);
     \node at (a) [below = 0mm] {${a= (0,1,0)^T}$};
     \draw [shorten <=0.1cm, shorten >=0.1cm] (a) -- (3,0);
    
     \coordinate (b) at (5,3);
     \node at (b) [below = 0mm] {${b(t)= (0,1,(t-\frac 1 2)^3)^T}$};
    
    \draw (-0.3,2) -- (7,2);
    \draw (7,2) to[out=0,in=135] (inf);
    \draw (-0.3,1) -- (7,1);
    \draw (7,1) to[out=0,in=135] (inf);
    \end{tikzpicture} 
}
\caption{Intersection of two lines with a removable singularity for $t = \frac 1 2$.}
\label{sing_basics:para_line_fig}
    \end{figure}
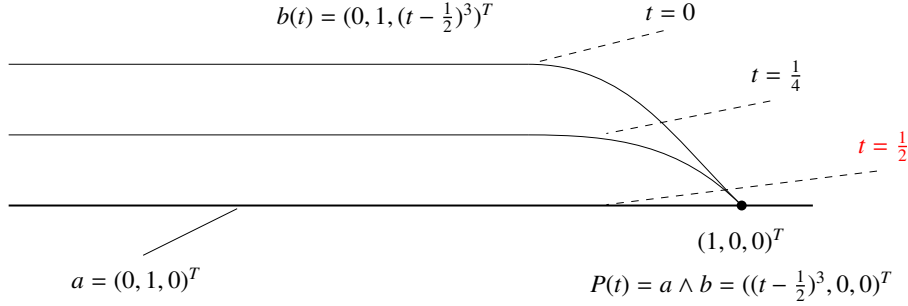
\end{example}
\begin{example}[Rotating Line]
        Consider the lines $a: = (1,1,0)^T$ and $b(t):= (1, -\frac{2}{\pi^2} \cdot(\sin(\pi\cdot t) -1), 0)^T$ for $t \in [0,1]$. Their point of intersection is $P(t):= a \wedge b = (0,0,  -\frac{2}{\pi^2} \cdot(\sin(\pi\cdot t) -1))^T$ with a removable singularity at $t = \tfrac 1 2$. 


        Define $f(t):=  -\frac{2}{\pi^2}( \sin(\pi\cdot t) -1)$ then we can expand the series around $t = \tfrac 1 2$: 
        \[
                f(t) = (t - \tfrac{1}{2})^2 + O((t - \tfrac{1}{2})^4)
        \]
        When we apply \Cref{sing_basics:ex_path} and normalize the term $P(t)$ by its $z$--component, we find a removable singularity in the classical complex analysis setting: we divide $f(t)$ by itself. The term $\tfrac f f$ has a singularity at $t=\tfrac 1 2$. We can apply complex de L'Hospital twice, either to the series expansion or the terms itself, which both removes the singularity:  $\tfrac{f''(\tfrac 1 2)}{f''(\tfrac 1 2)} = \tfrac 2 2 = 1$. Assembling the vector back together we can continuously extend $P$ at $t = \tfrac 1 2$ with $(0,0,1)^T$.

       More directly we can apply \Cref{sing:kth_deriv} to the components and derive the function:
        \[
                P'(\tfrac 1 2) = (0,0,-\tfrac{2}{\pi}\cdot \cos(\tfrac{\pi}{2}))^T = (0,0,0)^T, \quad P''(\tfrac 1 2) = (0,0,2\cdot \sin(\tfrac{\pi}{2})) \sim (0,0,1)^T.
        \]
        Both ways allow to define a $C^0$--continuation $[\hat P(t) ]$ for $[P(t)]$ at $t = \frac 1 2$:
        \[
            [\hat P(\tfrac 1 2)] = (0,0,1)^T.
    \]
\end{example}
\begin{example}[Midpoint of a Circle]
    Let $A = (-1,0,1)^T, B = (1,0,1)^T$ and $C = (0,1,1)^T$ be points in $\RP^2$. Define another point $C' =
    (0,0,1)^T$. Then we can define two circles (interpreted as conic section) being incident to 
    $A,B,C$ or $A,B,C'$ respectively, and call their associated conic matrices $X$ and $Y$. Spelling that out yields:

    \[
    X = \begin{pmatrix} 1&0&0 \\ 0&1&0 \\ 0&0&-1 \end{pmatrix} \qquad 
    Y = \begin{pmatrix} 0&0&0 \\ 0&0&1 \\ 0&1&0 \end{pmatrix}.
    \]
    Since $A,B$ and $C'$ are collinear the circle $\mathcal C_Y$ will degenerate to a line.

    Consider now the linear combination, $f(t):= (1-t)\cdot X + t\cdot Y$ for $t \in [0,2]$:
    this describes the unit circle for $t=0$, a line for $t=1$ and a 
    linear interpolation of the circle and the line (which is still a circle) for $t \in (0,1)$.
    If we calculate the center $M(t)$ of the conic given by $\mathcal C_{f(t)}$ one finds the following:
    \[
        M(t) = \begin{pmatrix}
            0\\(t-1)t\\ (t-1)^2
        \end{pmatrix} 
        = (t-1) \begin{pmatrix}
            0\\t\\ (t-1)
        \end{pmatrix} 
    \]
    which yields $(0,0,0)^T$ for $t=1$, which means that there is a (removable) singularity.
    We can again apply \Cref{sing:kth_deriv} and take the derivatives of $M(t)$: if one takes the first
    derivative, one recovers the correct solution even for $t=1$ which is
    $(0,1,0)^T$ the far point of the $y$ axis. This can of course be again interpreted as a continuous movement with a $C^0$--continuation at $t=1$. See \Cref{sing:circle_at_inf} for a picture. 
    \begin{figure}
        \centering
        \begin{tikzpicture}[baseline]
\coordinate (A) at (-1,0);
\fill (A) circle [radius=2pt];
\node at (A) [above = 2mm] {$\mathbf{A}$};

\coordinate (B) at (1,0);
\fill (B) circle [radius=2pt];
\node at (B) [above = 2mm] {$\mathbf{B}$};

\coordinate (C) at (0,0.5);
\fill (C) circle [radius=2pt];
\node at (C) [above = 2mm] {$\mathbf{C}$};

\coordinate (M) at (0,-0.75);
\fill (M) circle [radius=2pt];
\node at (M) [right = 2mm] {$\mathbf{M}$};

\draw [->] (C) -- (0,0.1);
\draw [->] (M) -- (0,-1.75);

\draw (0,-0.75) circle (1.25cm);
\end{tikzpicture}
\qquad
\qquad
\qquad
\qquad
\qquad
\qquad
\begin{tikzpicture}[baseline]
\coordinate (A) at (-1,0);
\fill (A) circle [radius=2pt];
\node at (A) [above = 2mm] {$\mathbf{A}$};

\coordinate (B) at (1,0);
\fill (B) circle [radius=2pt];
\node at (B) [above = 2mm] {$\mathbf{B}$};

\coordinate (C) at (0,0);
\fill (C) circle [radius=2pt];
\node at (C) [above = 2mm] {$\mathbf{C}$};

\coordinate (M) at (0,-3);
\fill[opacity=0.5] (M) circle [radius=2pt];
\node at (M) [below = 2mm,opacity=0.5] {$\mathbf{M}$};

\draw [shorten >=-2cm,shorten <=-2cm ] (-1,-3) -- (1,-3);
\coordinate (linf) at (3,-3.5);
\draw[line width=0.05mm] (linf) -- (2,-3.1);
\node at (linf) [below = 0.1mm,right=0.1mm] {$l_\infty$};

\draw [shorten >=-2cm,shorten <=-2cm ] (A) -- (B);

\end{tikzpicture}
\caption{\textbf{Degenerate Midpoint:} moving $C$ to the connecting line of $A$ and $B$ yields a removable singularity along a fixed path.}
       \label{sing:circle_at_inf}
    \end{figure}
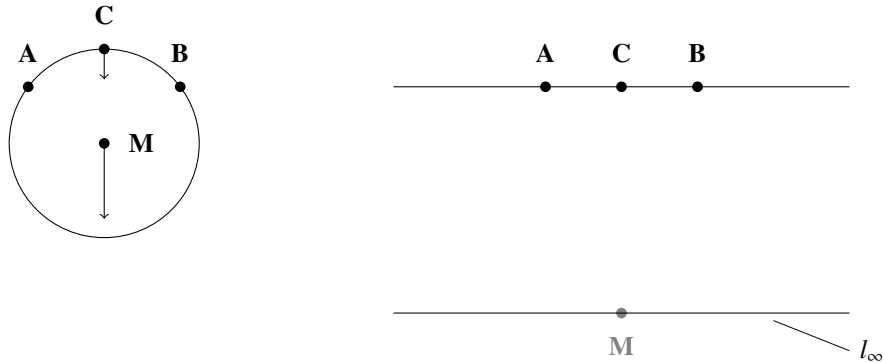
\end{example}
\begin{remark}
    For a practical implementation one could use so called ``automatic differentiation''. This combines the benefits of numerical and symbolical differentiation: fast execution and stability. We refer the reader to \cite{griewank2008evaluating}.
\end{remark}
\section{The Radical of all Evil}
In the last section we investigated the role of derivatives in the process of removing singularities. The basic idea was an unpretentious one: singularities of the form $0/0$ can be resolved using de L'Hospital's rule. All the theorems had one prerequisite: the analyticity of the functions around $t_0$. In geometry, a lot of constructions are accessible using only ruler constructions. These can be carried out analytic everywhere, but not without reason the classic constructions are often called ``ruler and compass constructions''. And the utilization of a compass means that algebraically one has to admit radical expressions (\cite{martin2012geometric}, \page 35). The complex radical functions is intrinsically monodromic and has a branch point at $0$ (and one at $\infty$). But more down to earth: complex radical functions cannot be defined analytic at branch points, which is the reason why \Cref{sing_basics:removal_ana} or \Cref{sing:kth_deriv} cannot be applied, since analyticity was a crucial premiss for these. Eventually this leads to the beautiful theory of Riemann surfaces. 

We will illustrate this with an example taken from \cite{kortenkamp1999foundations}.
\begin{example}[Disjoint Circle Intersection]
    \label{sing:expl_sqrt}
    \begin{figure}
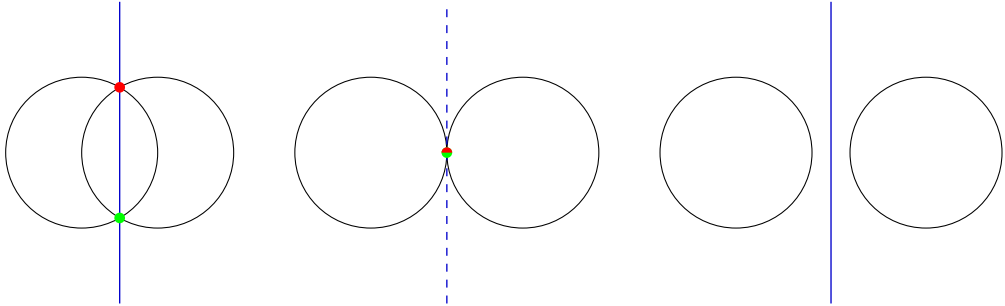

       \drawMcirc
       \caption{\textbf{Disjoint Circle Intersection:} The connecting line of the intersection of two circles. The line is undefined for the tangential situation in the middle.}
       \label{sing:circle_fig}
    \end{figure}
    As circles are conic sections, in general two of them have four points of intersection. It is well known that two points are always incident to all circles in $\CP^2$: namely the points $\II = (-\ii, 1,0)^T$ and $\JJ = (\ii, 1, 0)^T$. We will neglect these two and concentrate on the remaining two points, which are usually considered in Euclidean geometry. It might also occur that the points of intersection become complex. The details and an algorithm how to implement the intersection can be found in \cite{richter2011perspectives} \page 196 ff..

    Usually these two ``other points'' (\ie not being $\II$ or $\JJ$) of intersection are disjoint, but there are two exceptions: if the two circles are the same and if the two circles are tangential. 

    We will now analyze the latter: the basic idea is to reduce the problem of intersecting two circles to the task of intersection a line and a conic. Intersecting a conic and a line involves using a square root operation,
    since essentially we are solving a quadratic equation (\cite{richter2011perspectives} \page 194 ff.). In the tangential situation the two points of intersection will merge to one. Seen from the point of view of quadratic equations: the discriminant in the solution of the quadratic expression vanishes, leaving only one solution. 

%
    Now if one examines the connecting line of those two merging points the operation will be inherently undefined. But actually the singularity is removable: if one considers a limit process one can assign a unique continuous extension of the function. We will now examine the situation along a fixed path.

    Given two circles $C,D$ with radius 1 and centers $M_C = (0,0,1)^T$ and $M_D(t):= (t,0,1)^T$ for $t \in (0,2]$. $D$ is moving along the $x$-axis and one obtains two
    points of intersection of $C$ and $D$, call them $A$ and $B$:
    \[
        A = \begin{pmatrix}  t \\ \sqrt{1-t^2} \\ 1 \end{pmatrix}, \quad
        B = \begin{pmatrix}  t \\ -\sqrt{1-t^2} \\ 1 \end{pmatrix},
    \]
    We will consider the connecting line of $A$ and $B$ and call this line $l$:
    \[
        l(t) =  \begin{pmatrix}  2\sqrt{1-t^2} \\ 0  \\ 2t \sqrt{1-t^2}
        \end{pmatrix} = (2\sqrt{1-t^2} )\begin{pmatrix}1   \\ 0  \\ t
        \end{pmatrix} \sim  \sqrt{1-t^2} \begin{pmatrix}1   \\ 0  \\ t  \end{pmatrix}
    \]
    A drawing of the situation can be found in \Cref{sing:circle_fig}. The connecting line $l$ is well defined for all $t$ expect for $t=1$:
    there $l(t)$ has a singularity (and also at $t = -1$, which is not part of the domain of the function). The theory we developed so far would suggest to apply de L'Hospital's rule or similar techniques based on derivatives. So if we apply 
    \Cref{sing:kth_deriv} we find the following: 
    \begin{align*}
        l'(t) &=  \begin{pmatrix} \frac{-t}{\sqrt{1-t^2}}   \\ 0  \\
            \frac{1-2t^2}{\sqrt{1-t^2}} \end{pmatrix} =
        \frac{1}{\sqrt{1-t^2}}\begin{pmatrix} -t \\ 0 \\ 1-2t^2 \end{pmatrix}\\
        l''(t) &= \begin{pmatrix} \frac{-1}{(1-t^2)^\frac{3}{2}}   \\ 0  \\
            \frac{t(2t^2-3)}{(1-t^2)^\frac{3}{2}} \end{pmatrix} = 
        \frac{1}{(1-t^2)^\frac{3}{2}} \begin{pmatrix}  -1 \\ 0 \\ t(2t^2-3)\\
        \end{pmatrix}\\
        l'''(t) &=  \dots
    \end{align*}
    We run into the same singularity over and over again. This is of course not remarkable at all since the derivates of the complex square radical have an essential singularity at $0$.
\end{example}
\begin{remark}
    Note that we excluded $t=0$ in at which the two circles would merge. We will analyze the difference later on and argue that there are situations where even the merging of two circles can yield a meaningful connecting line but only for a given path. But the solution is \textbf{only} meaningful along this path and behaves discontinuous over spacial perturbations. 
\end{remark}
So we can see that derivatives cannot resolve the singularity. We have to come up with different techniques to handle the situation. As mentioned before the singularities are isolated and so a slight perturbation of the function would lead to a result in proximity to the desired solution. Unfortunately, simply perturbing the construction by a real number will introduce numerical error that, depending on the function, can grow very quickly. Hence, it would be beneficial if the perturbations were smaller than every real number, \ie infinitesimal. Therefore, we decided to integrate a non-Archimedean field, the so called hyperreal and hypercomplex numbers, into (dynamic) projective geometry. These fields are well known from the theory of non-standard analysis and contain infinitesimal and unlimited numbers.

\section{Non-standard Analysis}
\label{chap:nsa}
An important step towards rigorous treatment of infinitesimal quantities was the article ``\textsc{Eine Erweiterung der Infinitesimalrechnung}'': \cite{schmieden1958erweiterung}. Their construction had the drawback that it wasn't actually a field and therefore contains elements with no multiplicative inverse. Nevertheless, this article was a huge step towards the formal treatment of the infinitesimals.

Eventually a thorough treatment of non-standard analysis was given by \cite{robinson1961non} which overcame the drawbacks of the ansatz by Laugwitz and Schmieden and constructed a proper field with infinitesimal and unlimited members.

We will not admit how to construct the hyperreal or hypercomplex numbers and refer to the literature: \cite{goldblatt,robinson1961non}.

\begin{notation}[$\Rs$ and $\Cs$, \cite{goldblatt} \page 25]
  We denote by $\Rs$ and $\Cs$ the \textit{hyperreal} and the \textit{hypercomplex} numbers.
\end{notation}

\begin{notation}[Enlargement, \cite{goldblatt} \page 28]
  For a set $A \subset \R$ or $A \subset \C$ we denote by $\mathit{\As}$ the \textit{enlargement of $\mathit{A}$}. 
\end{notation}
\begin{remark}
  The enlargement operation adds all non-standard members to a set (\eg infinitesimal and unlimited numbers). We will not got into detail here and refer to \cite{goldblatt}.
\end{remark}
  \begin{intuition}[Extended Function, \cite{goldblatt} \page 30]
    Let $f: A \subset \C \rightarrow \C$ be a function. We denote by $f^*: \As \rightarrow \Cs$ the \textit{extended function $\mathit{f}$}. It holds true that $f(z) = f^*(z) \; \forall z \in A$.
  \end{intuition}
  \begin{intuition}[Universal Transfer, \cite{goldblatt} \page 45]
    If a property holds true for all real (complex) numbers, then it holds true for all hyperreal (hypercomplex) numbers. 
  \end{intuition}
  \begin{theorem}[Existential Transfer, \cite{goldblatt} \page 45]
    If there exists a hyperreal (hypercomplex) number satisfying a certain property, then there exists a real (complex) number with this property.
  \end{theorem}
  \begin{remark}
    We will not give a formal definition of the transfer principle and the star transform (transferring a logical statement from the non-standard world forth and back) and refer to \cite{goldblatt}.
  \end{remark}
\begin{definition}[$\Ks$ sets, partially \cite{goldblatt} \page 50]
    \begin{align*}
    \mathbb{I}_\R &:= \{\epsilon \in \Rs: \, |\epsilon| < |r| \quad \forall r \in \R \},\quad \textit{real infinitesimal numbers}\\
    \mathbb{I}_\C &:= \{\epsilon \in \Cs: \, |\epsilon| < |r| \quad \forall r \in \R \},\quad \textit{complex infinitesimal numbers}\\
        \mathbb{A}_\R&:= \{r^* \in \Rs: \exists n \in \N: \frac{1}{n} < |r^*| <
n\},\quad \textit{real appreciable numbers}\\
        \mathbb{A}_\C&:= \{r^* \in \Cs: \exists n \in \N: \frac{1}{n} < |r^*| <
n\},\quad \textit{complex appreciable numbers}\\
        \mathbb{R_\infty^+}&:= \{H \in \Rs: \, H > r \quad \forall r \in \R
\},\quad \textit{positive unlimited numbers}\\
        \mathbb{R_\infty^-}&:= \{H \in \Rs: \, H < r \quad \forall r \in \R
        \},\quad \textit{negative unlimited numbers}\\
        \mathbb{R_\infty}&:= \R_\infty^+ \cup \R_\infty^- ,\quad \textit{real unlimited numbers}\\
	\mathbb{C_\infty}&:= \Cs \setminus \{ \mathbb{I}_\C \cup \mathbb{A}_\C \}  ,\quad \textit{complex unlimited numbers}\\
        \mathbb{L}_\R&:= \Rs \setminus \R_\infty    ,\quad \textit{real limited numbers}\\
        \mathbb{L}_\C&:= \Cs \setminus \C_\infty    ,\quad \textit{complex limited numbers}
\end{align*}
\end{definition}
\begin{remark}
  We will drop the index $\R$ or $\C$ of $\mathbb{I}_\R, \mathbb{I}_\C, \mathbb{A}_\R,  \mathbb{L}_\R$ or $ \mathbb{L}_\C$ if the context admits. 
\end{remark}
\begin{definition}[Infinitely Close and Limited Distance, partially \cite{goldblatt} \page 52]
    Define for $b,c \in \Ks$ the equivalence
    relation \[ b \simeq c \text{ if and only if }, b-c \in \mathbb I \] and we
    call $b$ is \textit{infinitely close} to $c$. Furthermore, we write\[b
        \sim c \text{ if and only if },  b-c \in \mathbb L \]
        and we say $b$ has \textit{limited distance} to $c$.
\end{definition}

\begin{definition}[Halo and Galaxy, partially \cite{goldblatt} \page 52]
     Define the \textit{halo} of an arbitrary $b \in \Ks$ as 
    \[
        \hal(b) := \{c \in \Ks : b \simeq c\}
    \]
    and the \textit{galaxy} of $b \in \Ks$ as 
    \[
        \gal(b) := \{c \in \Ks : b \sim c\}.
    \]
\end{definition}
\begin{theorem}[Shadow]
    \label{nsa_basics:shadow}
    Every limited hyperreal (hypercomplex) $z^*$ is infinitely close to exactly one real (complex) number. We
    call this the \textit{shadow} of $z^*$ denoted by $\sh(z^*)$.
\end{theorem}
\begin{proof}
    The real case is proven in \cite{goldblatt} \page 53. We will prove the complex case using the constructions of the real case for both real and imaginary part. Then we use an estimate to conclude.

    Write $z^* = a + \ii \cdot b$, then we have $a,b \in \Lim_\R$ (otherwise $z^*$ would not be limited). Define the following sets:
    \[
        A:=\{r \in \R \bbar r < a\}, \quad B:=\{r \in \R \bbar r < b\}.
    \]
    And set $\alpha:= \sup A$ and $\beta:=\sup B$. By the completeness of $\R$ we know that $\alpha, \beta \in \R$. Let $z:= \alpha + \ii \cdot\beta$. We first show that $z^* - z \in \Lim_\C$: take any $\epsilon > 0$. Since $\alpha$ is an upper bound of $A$ we know $\alpha + \epsilon \notin A \Rightarrow a \leq \alpha + \epsilon$. Furthermore, $\alpha - \epsilon < a$, hence otherwise we have $a \leq \alpha - \epsilon$ which would be a lower upper bound of $A$ which is a contradiction to the construction of $\alpha$. So we have 
    \[
        \alpha - \epsilon < a \leq \alpha + \epsilon \Leftrightarrow |a-\alpha| \leq \epsilon.
    \]
We can argue the same way for $b$ and have $|b-\beta| \leq \epsilon'$ for some $\epsilon' >0$.
Finally, we can conclude 
\[
    |z^*-z| = \sqrt{(a-\alpha)^2 + (b-\beta)^2} \leq |a-\alpha| + |b-\beta| \leq \epsilon + \epsilon' =: \hat \epsilon,
\]
where we used that $\norm{x}_1 \geq \norm{x}_2 \forall x \in \R$ and by universal transfer also for all $x \in \Rs$.
Since this holds for all $\hat \epsilon > 0$ we know that $z^*$ and $z$ are infinitesimal close.

We still have to show the uniqueness of $z$. Assume there is another $z' \in \C$ with the same property. Then $z^* \simeq z'$ and therefore $z \simeq z'$. Since both $z$ and $z'$ are complex numbers this means $z=z'$.
\end{proof}

\begin{lemma}[Complex Shadow]
    \label{nsa_basics:complex_shadow}
   For $z \in \Lim_\C$ and $z = a + \ii \cdot b$ we have 
   \[
       \sh(z) = \sh(a) + \ii \cdot \sh(b)
   \]
\end{lemma}
\begin{proof}
    Direct consequence of the construction in the proof of \Cref{nsa_basics:shadow}.
\end{proof}
\begin{lemma}[Complex shadow and conjugation]
    \label{nsa_basics:conj_shadow}
    Let $z \in \Lim$, then it holds true that $\sh(\conj{z}) = \conj{\sh(z)}$.
\end{lemma}
\begin{proof}
    Apply \Cref{nsa_basics:complex_shadow}. 
\end{proof}
\begin{definition}[Almost real]
    We call a number $z \in \Cs$ \textit{almost real}, if $z$ is infinitely close to a real number. That means there exists a $r \in \R : z \simeq r$. 
\end{definition}
\begin{lemma}[Shadow properties]
    \label{nsa_basics:sh_prop}
Let $a,b \in \mathbb L$ and $n \in \mathbb \N$ then
\begin{enumerate}
        \item $\sh(a \pm b) = \sh(a) \pm \sh(b)$
        \item $\sh(a \cdot b) = \sh(a) \cdot \sh(b)$
        \item $\sh(\frac{a}{b}) = \frac{\sh(a)}{\sh(b)}, \text{ if }\sh(b) \neq
            0$
        \item $\sh(b^n )= \sh(b)^n$
        \item $\sh(|b|)= |\sh(b)|$
        \item for $a,b \in \Lim_\R:$ $\text{if } a \leq b \text{ then } \sh(a) \leq \sh(b)$
    \end{enumerate}
\end{lemma}
\begin{proof}
    For the real case see \cite{goldblatt} \page 53-54. For the complex case we can reduce this to the real case by \Cref{nsa_basics:complex_shadow}.
\end{proof}

\begin{theorem}[{Isomorphism{, partially \cite{goldblatt} \page 54}}]
    \label{nsa_basics:iso}
    The quotient ring $\mathbb{L_\R}/\mathbb{I_\R} \; (\mathbb{L_\C}/\mathbb{I_\C})$ is isomorphic to the field of the
    real (complex) numbers by $\hal(b) \mapsto \sh(b)$. Therefore $\mathbb I$ is a
    maximal ideal of the ring $\mathbb L$.
\end{theorem}
\begin{proof}
 For the real case \cite{goldblatt} \page 54. For the complex case, note that $\Cs \simeq \Rs \times \Rs$ and apply the same arguments as for the real case.
\end{proof}
\begin{theorem}[{Continuity{, partially \cite{goldblatt} \page 75}}]
    \label{nsa_basics:conti}
The function $f: \C \rightarrow \C$ is continuous at $c\in \C$, if and only if $f(c) \simeq f(x)$ for all $x
    \in \Cs$ such that $x \simeq c$. In other words if and only if 
    \[
        f(\hal(c)) \subset \hal(f(c)).
    \]
\end{theorem}
\begin{proof}
    The real case is again in \cite{goldblatt} \page 75. The complex case reads analogously since the definition of continuity is essentially the same only with a different definition of the absolute value.
\end{proof}
\begin{remark}
    This property will be crucial later on to resolve singularities in geometric constructions.
\end{remark}
\begin{lemma}[{Real Limits{, \cite{goldblatt} \page 78}}]
        \label{nsa_basics:real_limits}
    For $c, L \in \R$ and $f$ be defined on $A\subset \R$ then it holds true:
    \begin{align*}
        \lim_{x \rightarrow c} f(x) &= L \Leftrightarrow f(x) \simeq L \quad \forall x \in \As: x \simeq c,\; x \neq c\\
        \lim_{x \rightarrow c^+} f(x) &= L \Leftrightarrow f(x) \simeq L \quad \forall x \in \As: x \simeq c,\; x > c\\
        \lim_{x \rightarrow c^-} f(x) &= L \Leftrightarrow f(x) \simeq L \quad \forall x \in \As: x \simeq c,\; x < c\\
    \end{align*}
\end{lemma}
\begin{lemma}[Squeezing Limits]
    \label{nsa_basics:sqeezing_limits}
    Let $A\subset \R, c \in \inter(A)$, $L \in \C$ and let $f: A \setminus \{c \} \rightarrow \C$ be continuous. 
    If there exists a $\Dx \in \I_\R \Smz, \Dx > 0$ such that $f(c+\Dx) \simeq f(c-\Dx) \simeq L$, then the function $f$ can be continuously extended on $A$ with $f(c) = L$.
\end{lemma}
\begin{proof}
    Define the sets $A^+ = \{a \in A \bbar a > c\}$, $A^- = \{a \in A \bbar a < c\}$ and $H^\circ:= \hal(c) \setminus \{c\}$. Since $f$ is continuous it holds true by \Cref{nsa_basics:real_limits} 
    for all $h^+ \in f(H^\circ \cap (A^+)^*)$ that $h^+ \simeq f(c+\Dx)$ and furthermore for $h^- \in f(H^\circ \cap (A^-)^*)$ that $h^- \simeq f(c-\Dx)$. As by assumption $f(c+\Dx) \simeq f(c-\Dx)$ it holds true that $h \simeq f(c\pm \Dx)$ for all $h \in f(H^\circ)$. 
    
    Now we are almost done: it holds true that $(A\setminus \{c\})^* = \As \setminus \{c\}$ (see \cite{goldblatt}, \page~29). So if we want to extend $f$ to whole $\As$ we only have to define a value for $f(c)$: we define $f(c):= L$.
    Since $f(c\pm \Dx) \simeq L = f(c)$ it holds true that $f(x) \simeq L$ for all $x \in \hal(c)$. That is equivalent with the continuity of $f$ at $c$ by \Cref{nsa_basics:conti}. 
\end{proof}
\begin{remark}
    The last lemma is very useful since it gives us a recipe to continuously extend a function! If we have a discontinuity of $f$ at some point $c$, it is removable if and only if the shadows of $f(c+\Dx)$ and $f(c-\Dx)$ coincide and we can continue the function with the calculated shadow. 
\end{remark}
\begin{lemma}[{Complex Limits{, partially \cite{goldblatt} \page 78}}]
        \label{nsa_basics:limits}
    Let $c, L \in \C$ and $f$ be defined on $A\subset \C$. Then 
    \[
        \lim_{x \rightarrow c} f(x) = L \Leftrightarrow f(x) \simeq L \quad \forall x \in \As: x \simeq c,\; x \neq c.
    \]
\end{lemma}
\begin{proof}
    Essentially, the proof is based on the equivalence of the Weierstrass $\epsilon-\delta$ continuity and the limit definition of continuity. The theorem is stated for the real case in \cite{goldblatt} \page 78. The proof by transfer can be found on \page 75--76. The complex version reads completely analogously.
\end{proof}
The next theorem will give important facts about the topology. It turns out that a set is open if and only if it includes the halo of all its points.

Remember that a set $A$ is open if and only if $\inter A = A$ (see \cite{munkres2000topology} \page 95).
\begin{lemma}[Halo and a Metric, \cite{davis1977applied} \page 89]
  \label{halo_and_matric}
 If $d$ is a standard metric on $X$ then it holds true 
    \[
        \hal(x):= \{ y \in X \bbar d(x,y) \simeq 0 \}.
    \]
   And so we can write $q \simeq p$ for $p,q \in X^*$ if and only if $d(p,q) \simeq 0$.
\end{lemma}
\begin{theorem}[Topological Non-standard Continuity, \cite{davis1977applied} \page 79]
    \label{nsa_basics:nsa_topo_conti}
   Let $f$ map $X$ to $Y$ where $X,Y$ are both topological spaces. Let $p \in X$. Then $f$ is continuous at $p$ if and only if 
   \[
       q \simeq p \Rightarrow f(q) \simeq f(p).
   \]
\end{theorem}
\begin{remark}
    The characterization of continuity in the previous theorem reads exactly analogously to the definition of continuity in \Cref{nsa_basics:conti} and this is completely canonical since the previous theorem generalizes infinitesimal proximity to arbitrary topological spaces.
This can be used later on to define continuity directly in the projective space.
\end{remark}
%
\begin{lemma}[Constant on Open Sets]
    \label{nsa_basics:const_fct}
    Let $A\subset \C$ open, $c\in A$ and $f: A \rightarrow \C$ be a function on $A$. 
    \begin{enumerate}
\item If $f$ is constant on $\hal(c)$ then there is an open set $A' \subset A$ with $c \in A'$ such that $f$ is constant on $A'$.
\item If $f$ is constant on on $A$ then $f$ is also constant on $\As$.
    \end{enumerate}
\end{lemma}
\begin{proof}
We proof 1: if $f$ is constant on $\hal(c)$ the following statement is true:
   \[
       \exists \delta \in (\R^+)^*: \forall x \in \Cs:\left(|x-c| < \delta \Rightarrow f(x) = f(c)\right).
   \]
   Now we apply existential transfer and find the next statement also true:
   \[
       \exists \delta \in \R^+: \forall x \in \C:\left(|x-c| < \delta \Rightarrow f(x) = f(c))\right.
   \]
   Which says that there is an open set $A':=  \{x \in A \bbar |x-c| < \delta \}$ such that $f$ is constant on $A'$.

   Now we proof 2: if $f$ is constant on $A$, then the following sentence is true: there is $d \in \C$ such that
   \[ \forall x \in A: f(x) = d 
   \]
   Then apply universal transfer:
   \[ \forall x \in \As: f(x) = d
   \]
which is the claim.
\end{proof}
\begin{remark}
    As the converse argument this also implies that if a function $f$ does not vanish on the $\hal(c)$ then there is an $\epsilon >0, \epsilon \in \R$ such that the function $f$ also does not vanish on $B_\epsilon(c)$.
\end{remark}
\begin{lemma}[Constant]
    Let $D$ be a domain and $z_0 \in D$. If an analytic function $f: D \rightarrow \C$ is constant on $\hal(z_0)$ then it constant on $D$.
\end{lemma}
\begin{proof}
  By \Cref{nsa_basics:const_fct} we know that if $f$ is constant on a $\hal(z_0)$ then there is an open set $D' \subset D$ where $f$ is constant. By the identity theorem for complex functions (see \cite{ablowitz2003complex} \page 122) the claim follows immediately.
\end{proof}

\section{Non-standard Projective Geometry}
\label{chap:nsa-pg}

We will now introduce the usage of hyperreal and hypercomplex numbers in projective geometry. Surprisingly there is some preliminary work by A.\ Leitner (\cite{leitner2015limits}). Leitner introduces projective space over the hyperreal numbers (not considering complex space) and uses this to proof properties of the diagonal Cartan subgroup of $\operatorname{SL}_n(\mathbb{R})$.
\begin{definition}[$\KsP^d$]
Let $d \in \N$. We define $\KsP^d$ analogously to $\KP^d$:
    \[
        \KsP^d \; := \;  \dfrac{\Ks^d \setminus \{0\}}{\Ks \setminus \{0\}}
    \]
    where $\Ks$ denote the hyper (real or complex) numbers.
\end{definition}
\begin{definition}[Standard and Non-standard Projective Geometry]
    If we talk about projective geometry in $\KP^d$, we will refer to it as \textit{standard projective geometry} and for results in the extended field $\KsP^d$ we will refer to as \textit{non-standard projective geometry}.
\end{definition}
\begin{definition}[Similary]
    If two vectors $x,y \in \KP^2$ represent the same equivalence class, we will write $x \sim_{\KP^2} y$ and analogously for the enlarged space $\KsP^2$ we will write $x \sim_{\KsP^2} y$.
    \end{definition}
    \begin{remark}
       If the situation allows, we will sometimes also drop the ``$\sim$'' relation symbol and just write ``$=$''. 
    \end{remark}
    \begin{definition}[Representatives]
            \label{nsa_pg:representative}
            We call a representative $x$ of $[x] \in \KsP^d$ 
        \begin{itemize}
        \item    \textit{limited}, if and only if for all components $x_i$ of $x$ it holds true that $x_i \in \mathbb L$. 
        \item \textit{infinitesimal}, if and only if for all
    components $x_i$ of $x$ it holds true that $x_i \in \mathbb I$. 
\item    \textit{appreciable}, if and only if for all components $x_i$ of $x$ it holds true that $x_i \in \mathbb L$ and there is at least one component $x_{i'}$ which is appreciable. 
    \item \textit{unlimited}, if and only if for at least one
        $x_i$ of $x$ it holds true that $x_i \in \mathbb \K_\infty$. 
    \end{itemize}
\end{definition}
\begin{remark}
        We will denote by $\norm{x}$ the Euclidean norm $\norm x_2$ of a vector in $x \in \K^d$ or $x \in \Ks^d$, if not stated otherwise. 
\end{remark}
%
%
\begin{definition}[Points and Lines in $\KsP^2$]
    We define the sets of point $\mathcal{P}_\Ks$ and lines $\mathcal{L}_\Ks$ of $\KsP^2$ by
    \begin{align*}
        \mathcal{P}_\Ks \; := \;  \dfrac{\Ks^3 \setminus \{0\}}{\Ks \setminus
        \{0\}}, \quad 
        \mathcal{L}_\Ks \; := \;  \dfrac{\Ks^3 \setminus \{0\}}{\Ks \setminus \{0\}}.
    \end{align*}
\end{definition}
This definition is is the logical consequence of the introduction of hyperreal and hypercomplex numbers. But this also entails modifications to the standard operations in projective geometry. For example the cross product of two points defines the connecting line, but the definition will not be independent of the representative any more. This is not as bad as it seems on the first sight: it turns out that if representatives are of reasonable length (meaning appreciable), a lot of properties transfer from standard geometry to the non-standard version. Therefore, we will introduce adapted definitions for the basic operations in projective geometry.

\begin{definition}[Appreciable Representative]
   We call $x = (x_1,\ldots,x_d)^T$ a \textit{appreciable
   representative} of $[x] \in \CsP^d$ if the norm of $x$ is appreciable, \ie $\norm{x} \in \A$.
    \end{definition}
\begin{definition}[Almost Equivalent]
    We call two objects $x,y \in \CsP^n$ \textit{almost equivalent}, if
    $\exists \lambda \in \Cs$ such that
    \[
        x_\A \simeq \lambda \cdot y_\A.
    \]
  \end{definition}
    We write again $x \simeq y$. This is obviously consistent with any metric in $\CsP^2$ as in \Cref{halo_and_matric}.
\begin{definition}[Projective Shadow]
    \label{nsa_pg:psh}
    The \textit{projective shadow} of $[x] = [(x_1,\ldots,x_d)^T] \in \CsP^d$ is defined by
    \[
        \psh ([x]) := [\sh (x_{\mathbb A})]    \]
        where $x_\A$ is an appreciable representative of $x$.
\end{definition}

\begin{definition}[Appreciable Cross-Product]
    \label{nsa_pg:cross-prod}
    Let $x,y$ be in two vectors in $\CsP^2$. We define the \textit{appreciable cross-product} $\mathit{\times_*}:$ by:
   \[
       x \times_* y :=  x_\A \times y_\A 
   \]
   For appreciable representatives $x_\A, y_\A$ of $x$ and $y$.
\end{definition}
\begin{definition}[Almost Far Point]
    We call $p \in \mathcal{P}_\Cs$ an \textit{almost far point}, if it holds true for an appreciable representative $p_\A$ of $p$ with 
    \[
        p_\A = \begin{pmatrix} x \\ y \\ z\end{pmatrix} \quad \text{ and } x, y, z  \in \Lim
    \]
    that $z \in \I$, \ie that the $z$-component of the appreciable representative is infinitesimal.
\end{definition}
\begin{definition}[Appreciable join and meet]
    Analogously to the standard definition of the join in $\CP^2$ we define the \textit{appreciable $\mathit{join}$} for two points $p,q \in \mathcal{P}_\Cs$ by
        \[
            \join_*(p,q):= p \times_* q = p_\A \times q_\A.
        \]
        And for two lines $l,m \in  \mathcal{L}_\Cs$ we define the \textit{appreciable $\mathit{meet}$} by 
        \[
            \meet_*(l,m):= l \times_* m = l_\A \times m_\A.
        \]
\end{definition}
\begin{remark}
    The result of an appreciable $\join$ or $\meet$ is not necessarily appreciable! Take for example the two points $p = (0,0,1)$ and $q = (\epsilon, 0,1)$, with $\epsilon \in \I$, \ie $\epsilon$ is infinitesimal. Then their appreciable $\join$, which is equivalent to the standard join since $p$ and $q$ are appreciable, is 
    \[
    \join_*(p, q) = \join(p,q) = \begin{pmatrix} 0 \\ \epsilon \\ 0 \end{pmatrix}
    \]
    which is not an appreciable vector!
\end{remark}

\begin{definition}[Almost Orthogonal]
    We two vectors $l,p \in \Cs^3$ \textit{almost orthogonal} and write $[p] \perp_\I [l] $  if and only if
    \[
 \langle p, l \rangle_*       = \epsilon \in
        \mathbb{I}.
    \]
\end{definition}
\begin{definition}[Almost Incident]
    Define relation $\mathcal{I}_\Cs \subset \mathcal{P}_\Cs \times \mathcal{L}_\Cs$ for a point $p \in  \mathcal{P}_\Cs$ and a line $l \in  \mathcal{L}_\Cs$ by
    \[
        [p] \mathcal{I}_\Cs [l] \Leftrightarrow  [p] \perp_\I [l].
    \]
    Then we call $p$ \textit{almost incident} to $l$.
\end{definition}

\section{Non-standard Analysis and Removal of Singularities}
Our original motivation to study non-standard analysis was the promise to resolve singularities in geometric constructions. So we follow up to that promise and will provide the reader with methods to do so. We will now combine non-standard analysis and non-standard projective geometry to achieve the goal.

Imagine the situation of an isolated removable singularity: on every epsilon ball around the singularity the function is well behaved and even analytic. By the notion of non-standard continuity every perturbation by in infinitesimal displacement will be infinitely close the analytic continuation of the function.
We will now analyze \Cref{sing_basics:ex_path} by \jrguk and its prerequisites from a non-standard viewpoint. The basic idea of the Theorem was that if a function $\Psi$ does not vanish on an epsilon ball around a singularity $t_0$ and furthermore if there is a component $\Psi_i$ of $\Psi$ such that the fraction $\frac{\Psi_j(t_0)}{\Psi_i(t_0)}$ has a removable singularity then we can continuously extend the function at $t_0$.
Firstly, we will generalize the non-vanishing property from an epsilon ball to the halo around $t_0$.
\begin{lemma}[Non-vanishing on the Halo]
        \label{nsa_prefac:not_vanish}
        \pathsettingwocomp

        Define $A'$ as the set of indices of $\Psi$--components which are not constantly 0 on $\hal(t_0) \setminus \{t_0\}$. Then the sets coincide: $A' = A$.
\end{lemma}
\begin{proof}
        First we show that $j \in A \Rightarrow j \in A'$:
       if $j \in A$ there is a fixed $\epsilon \in \R, \epsilon > 0$ such that the following sentence is true:
       \[
               \forall z \in \C: (|z-t_0| < \epsilon \wedge z \neq t_0 \Rightarrow \Psi_j(z) \neq 0).
       \]
       By universal transfer we know that
       \[
               \forall z \in \C^* : (|z-t_0| < \epsilon \wedge z \neq t_0 \Rightarrow \Psi_j^*(z) \neq 0).
       \]
       By definition the distance of every member $z'$ of $\hal(t_0)$ to $t_0$ is infinitesimal. Hence, every $z'$ trivially fulfills the bound $|z' - t_0| < \epsilon$. Therefore $f(z') \neq 0$, which is just differently phrased that $j \in A'$. 
       
  Now we show $j \in A' \Rightarrow j \in A$: we apply existential transfer: if $j \in A'$ the following sentence is true:
       \[
         \exists \epsilon \in (\R^+)^* \, \forall z \in \Cs : \abs{t_0 - z} < \epsilon \Rightarrow \abs{\Psi_j^*(z) }> 0.
       \]
       Then by existential transfer also the following sentence is true
       \[
         \exists \epsilon \in \R^+ \, \forall z \in \C : \abs{t_0 - z} < \epsilon \Rightarrow \abs{\Psi_j(z)} > 0.
       \]
       which means nothing else then $j \in A$.
\end{proof}
\begin{lemma}[Limited Fraction]
        \label{nsa_prefac:limited_frac}
        \pathsetting

        Then it holds true for all $\Delta t \in \I, \Delta t \neq 0$:
        \[
            \frac{\Psi_j^*(t_0+\Delta t)}{\Psi_i^*(t_0+\Delta t)} \in \Lim         \]
$\forall j \in \{1, \ldots, d\}$. In other words the fraction is limited and has a complex number as shadow.
\end{lemma}
\begin{proof}
    Firstly, we note that we do not divide by $0$ due to \Cref{nsa_prefac:not_vanish}. In \cite{kortenkamp2001grundlagen} it's argued that there is a $g \in \C$ such that 
        \[
                g = \lim_{t\rightarrow t_0} \frac{\Psi_j(t)}{\Psi_i(t)} \text{ for all } t \in B_\epsilon(t_0)
        \]
        Now we apply the rules on limits in \Cref{nsa_basics:limits}:
        \begin{align*}
                g &=  \lim_{t\rightarrow t_0} \frac{\Psi_j(t)}{\Psi_i(j)}\\ 
                \Leftrightarrow g &\simeq \frac{\Psi_j^*(t)}{\Psi_i^*(t)}, \quad \forall t \simeq t_0, t \neq t_0\\
                \Leftrightarrow g &\simeq \frac{\Psi_j^*(t_0 + \Dt)}{\Psi_i^*(t_0 +\Dt)}, \quad \forall \Dt \in \I, \Dt \neq 0\\
        \Rightarrow g &= \sh \left(\frac{\Psi_j^*(t_0 + \Dt)}{\Psi_i^*(t_0 +\Dt)} \right), \quad \forall \Dt \in \I, \Dt \neq 0        
\end{align*}
        Which is equivalent to the claim.
\end{proof}
Now we are able to formulate a first main result on the removal of singularities. The next theorem will give us a proper way to desingularize functions using methods of non--standard analysis.
\begin{theorem}[Shadow of a Singularity]
    \pathsettingwoprefactor

  Then for an arbitrary $\Dt \in \I \setminus \{0\}$ it holds true, for $j = 1, \ldots, d$:
        \begin{align*}
                \Theta_j(t_0) &= \sh \left(\frac{\Psi_j^*(t_0+\Delta t)}{\Psi_i^*(t_0+\Delta t)} \right).
        \end{align*}
\end{theorem}
\begin{remark}
        In other words we can continuously extend the function $\Psi$ at $t_0$ with 
        \begin{align*}
                \Theta_j(t_0) &= \sh \left(\frac{\Psi_j^*(t_0+\Delta t)}{\Psi_i^*(t_0+\Delta t)} \right).
        \end{align*}
\end{remark}
\begin{proof}
        By assumption $\Theta(t)$ is continuous, so we remember \Cref{nsa_basics:conti} which says that $\Theta(t_0) \simeq \Theta(t_0 + \Dt)$ for all $\Dt \in \I$. With this argument, choose $\Dt \neq 0$. For non-zero values of $\Dt$ \Cref{nsa_prefac:not_vanish} guarantees proper non-vanishing values.
        \begin{align*}
                \Theta_j(t_0) &\simeq \Theta_j(t_0 + \Dt) \overset{\text{Def.}}{=} \frac{\Psi_j^*(t_0+\Delta t)}{\Psi_i^*(t_0+\Delta t)}
        \end{align*}
        Now we apply \Cref{nsa_prefac:limited_frac}: 
        \begin{align*}
            \frac{\Psi_j^*(t_0+\Delta t)}{\Psi_i^*(t_0+\Delta t)} \in \Lim
            \Rightarrow \sh \left(  \frac{\Psi_j^*(t_0+\Delta t)}{\Psi_i^*(t_0+\Delta t)}\right) \in \C. 
        \end{align*}
        With the uniqueness theorem of shadows, \Cref{nsa_basics:shadow}, and the properties on limits, \Cref{nsa_basics:limits}, we can conclude that:
        \[
               \Theta_j(t_0) = \lim_{t \rightarrow t_0} \Theta_j(t)  =  \lim_{t\rightarrow t_0} \frac{\Psi_j(t)}{\Psi_j(t)} = \sh \left(\frac{\Psi_j^*(t_0+\Delta t)}{\Psi_i^*(t_0+\Delta t)}\right).
        \]
\end{proof}
\begin{example}[Circle Intersection Revisited]
    \label{nsa_prefac:circ_inter_revis}
        We will now analyze the circle intersection of \Cref{sing:expl_sqrt} with our new method. Essentially, the example reads like this: let $\Psi: [-2,2] \rightarrow \C$ be defined by
        \[
        \Psi(t):= \begin{pmatrix}  \sqrt{1-t^2} \\ 0  \\  \sqrt{1-t^2} \cdot t \end{pmatrix}
        \]
        The function $\Psi$ is singular for $t_0 = 1$ and $t_1 = -1$ since $\Psi(t_0) = \Psi(t_1) = 0$. Consider the situation for $t_0 = 1$:
        
Firstly, we observe that the premisses of \Cref{sing_basics:ex_path} hold true. In the notion of the theorem: $A = \{1,3\}$ and both $\Psi_i(t_0)/\Psi_j(t_0)$ and $\Psi_j(t_0)/\Psi_i(t_0)$ have a removable singularity, so the choice for the selected component $i$ is arbitrary. Without loss of generality choose $i = 3$. Then we have for $\Dt \in \I \Smz, t_0 = 1$:
        \begin{align*}
                \frac{\Psi_1^*(t_0+\Delta t)}{\Psi_3^*(t_0+\Delta t)} &= \frac{\Psi_1^*(1+\Delta t)}{\Psi_3^*(1+\Delta t)} = \frac{\sqrt{1-(1+\Dt)^2}}{\sqrt{1-(1+\Dt)^2}\cdot(1+\Dt)} \\
                                                                      &=\frac{\sqrt{-\Dt^2 - 2 \Dt}}{\sqrt{-\Dt^2 - 2 \Dt}\cdot(1+\Dt)} =\frac{1}{1+\Dt}
        \end{align*}
        Where we used that $\sqrt{-\Dt^2 - 2 \Dt}$ is a proper (infinitesimal) non-zero hypercomplex number, and we can simply cancel the fraction. 

        Then we apply the shadow:
        \[
            \Theta_1(t_0)=\sh \left(  \frac{\Psi_1^*(t_0+\Delta t)}{\Psi_3^*(t_0+\Delta t)}\right) =   \sh \left(\frac{1}{1+\Dt}\right) = \frac{1}{1} = 1.
        \]
        Analogously we can calculate the other components:
        \begin{align*}
            \Theta_2(t_0) &=  \sh \left(  \frac{\Psi_2^*(t_0+\Delta t)}{\Psi_3^*(t_0+\Delta t)} \right) =  \sh \left(\frac{0}{1+\Dt}\right) = 0\\ \Theta_3(t_0) &=\sh \left(  \frac{\Psi_3^*(t_0+\Delta t)}{\Psi_3^*(t_0+\Delta t)} \right) =  \sh \left(\frac{1+\Dt}{1+\Dt}\right) = 1.
        \end{align*}
        So the $C^0$--continuation of our function $\Psi$, which we denoted by $[\Theta]$ is given in the singularity $t_0 = 1$ by:
        \[
        [\Theta(t_0)] =  \begin{pmatrix} 1 \\ 0  \\  1  \end{pmatrix}
 \]
 and $[\Psi(t)]$ for all $t \neq 1$. A visualization of the example can be found in \Cref{fig:rem_sing_stable_circles}.
\begin{figure}
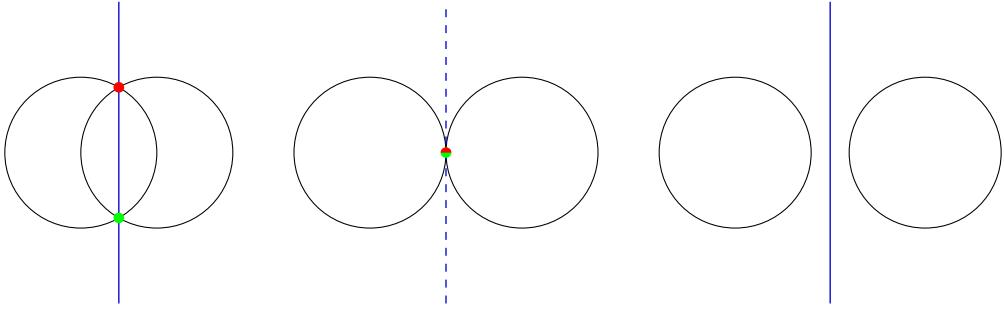

    \centering
\drawMcirc
\caption{$C^0$-continuation of the disjoint circle intersection.}
\label{fig:rem_sing_stable_circles}
\end{figure}
\end{example}
Next stop will be a more direct approach to removing singularities. While \Cref{sing_basics:ex_path} goes the way of dehomogenization there is a more direct way to obtain the desired desingularized function using hypercomplex numbers.
\begin{theorem}[Normalization and Continous Paths]
        \label{nsa_prefac:shadow_cont_abs}
        Let $\Psi: (0,1) \rightarrow \C^{d+1}$ be a continuous function and $t_0 \in (0,1)$. Let $\Psi$ fulfill the following premisses  
        \begin{enumerate}
            \item $\Psi(t)$ is not vanishing on $\hal(t_0) \setminus \{t_0\}$
            \item $ \exists L \in \C^{d+1} \Smz : L \simeq \frac{\Psi(t_0 + \Dt)}{\norm{\Psi(t_0 + \Dt)}} \simeq \frac{\Psi(t_0 + \Dt')}{\norm{\Psi(t_0 + \Dt')}}\quad \forall \Dt, \Dt' \in \I \Smz$.
        \end{enumerate}
        Then there exists $\epsilon \in \R, \epsilon > 0$ and a continuous function $\Theta(t): B_\epsilon(t_0) \rightarrow \C^{d+1} \Smz$ such that
        \[
            [\Psi(t)] = [\Theta(t)]
        \]
        for all $t \in B_\epsilon \setminus \{t_0\}$. 
        
        
This means that $[\Theta]$ is a $C^0$--continuation of $[\Psi]$ at $t_0$.
\end{theorem}
\begin{proof}
    We will only consider the case where $\Psi(t)$ vanishes at $t_0$. Otherwise the function $\Theta$ can be simply defined by $ \Theta(t) = \frac{\Psi(t)}{\norm{\Psi(t)}}$ and is continuous as quotient of continuous non-vanishing functions around $t_0$. The functions $[\Psi]$ and $[\Theta]$ coincide in projective space since $\norm{\Psi(t)}$ is a non-zero scalar factor for all $t \neq t_0$ and so the equivalence classes coincide in a projective setting.

    So let $\Psi$ vanish at $t_0$. Since the function does not vanish on $\hal(t_0) \setminus \{t_0\}$, by assumption 1), we can apply \Cref{nsa_basics:const_fct} which says that there is a $\R \ni \epsilon > 0$ such that $\Psi$ does not vanish on $B_\epsilon \setminus \{t_0\}$. So the function $\Theta(t): B_\epsilon \setminus \{t_0\} \rightarrow \C, t \mapsto \frac{\Psi(t)}{\norm{\Psi(t)}}$ is defined and as fraction of continuous functions again continuous. 

The continuous function $\Theta(t): B_\epsilon(t_0) \rightarrow \C^{d+1} \Smz$ is defined, for an arbitrary $\Dt \in \I \Smz$, as follows:
        \[
         \Theta(t):=\begin{cases}
             \frac{\Psi(t)}{\norm{\Psi(t)}}, &\text{ if } t \neq t_0\\
             \sh\left(\frac{\Psi(t+\Dt)}{\norm{\Psi(t+\Dt)}} \right), &\text{ if } t = t_0
        \end{cases}
        \]

    Now we can use \Cref{nsa_basics:limits}, the hypercomplex notion of limits, that says
    \begin{align*}
        &\lim_{t \rightarrow t_0} \Theta(t) = L \in \C, \; t \in B_\epsilon(t_0) \setminus \{t_0\}
        \\ &\Leftrightarrow \Theta(t) \simeq L \; \forall t \in (B_\epsilon)^*: t \simeq t_0,\; t \neq t_0
    \end{align*}
    which is equivalent to prerequisite 2) since:
    \begin{align*}
        & \Theta(t) \simeq L \; \forall t \simeq t_0, t \neq t_0  \\ \Leftrightarrow \; &\Theta(t) \simeq L  \; \forall t \in \hal(t_0) \setminus \{t_0\} \\ \Leftrightarrow \;  &\Theta(t_0+\Dt) = \Theta(t_0+ \Dt')  \; \forall \Dt, \Dt' \in \I \Smz 
    \end{align*}
    this means we can continuously extend the function $\Theta$ to $t_0$ with $\Theta(t_0) = L$. 

    Putting everything together: 
by construction $\Theta$ is continuous, with continuous extension $\Theta(t_0) = L$. Since the projection $[ \cdot ]: \C^{d+1} \rightarrow \CP^d$ is continuous by definition, the function $[\Theta]: B_\epsilon(t_0) \rightarrow \CP^d$ is continuous as composition of continuous functions. Since rescaling $\Psi$ by its (non-zero) norm does not change the equivalence class in a projective setting it holds true:
    \[
        [\Psi(t)] = [\Theta(t)] \quad \forall t \in B_\epsilon \setminus \{t_0\}
    \]
So $[\Theta]$ is a $C^0$--continuation of $[\Psi]$ at $t_0$.
\end{proof}
\begin{remark}
    In the previous theorem, the domain for $t_0$ can easily be extended from $(0,1)$ to its closure $[0,1]$ with the same arguments. For the sake of readability we omitted the generalization.
\end{remark}
\begin{definition}[Normalized Function]
    Let $t_0 \in \R$ and $f: B_\epsilon(t_0) \rightarrow \C^n \Smz$ be continuous. We call the function 
    \[
        \normalize{f}(t) := \frac{f(t)}{\norm{f(t)}}
    \]
    the \textit{normalization of $\mathit{f}$}.
\end{definition}
       \begin{remark}
           We usually define normalization via the Euclidean norm $\norm{\cdot}_2$ but actually it is not important which norm we use. By the equivalence of the norms in $\C^d$ we always will obtain appreciable representatives with unit length one (according to the chosen norm). For illustrating examples it is very practical to use the maximum norm and normalize the vector by the absolute value of the maximal element. As it turns out, we actually can also divide by the absolute value of the maximum but by the value itself! The norm of the maximal value and the maximal value itself only differ by the complex phase, which is an appreciable number. Therefore, we simply can pick the $\argmax$ of the maximum norm and divide by the maximal component.
       \end{remark}
\begin{corollary}[Projective Shadow and Normalized Functions]
    \label{nsa_pg:psh_normalize}
    Let $t_0 \in \R$ and $f: B_\epsilon(t_0) \rightarrow \C^{d+1} \Smz$ be continuous. Then it holds true:
    \[
        [\normalize{f}(t)] = \psh([f(t+\Dt)]) \; \forall \Dt \in \I
    \]
\end{corollary}
\begin{proof}
    Obvious by \Cref{nsa_pg:psh}.
\end{proof}
\begin{remark}
    This means that condition 2.\ of \Cref{nsa_prefac:shadow_cont_abs} 
    \[
             \exists L\in \C^{d+1}  : L \simeq \frac{\Psi(t_0 + \Dt)}{\norm{\Psi(t_0 + \Dt)}} \simeq \frac{\Psi(t_0 + \Dt')}{\norm{\Psi(t_0 + \Dt')}}\quad \forall \Dt, \Dt' \in \I \Smz
    \]
can be written as 
    \[
        \exists L \in \CP^{d} : [L] = \psh(\Psi(t_0 + \Dt)) = \psh(\Psi(t_0 + \Dt')) \quad \forall \Dt, \Dt' \in \I \Smz.
    \]
\end{remark}
Instead taking limits in the space $\C^{d+1}$ and then use a projection to $\CP^d$ one can also directly use the quotient space continuity.
\begin{remark}
    This is an important step towards a more algorithmic approach to resolve singularities. The theorem allows us to directly compute the singularity free representation of the continuous path. 
   If we know that a function is continuous up to a singular point we can use  \Cref{nsa_basics:sqeezing_limits} to check whether the right hand side and left hand side limits coincide with an arbitrary $\Delta t$. This gives a practical algorithm to test whether a singularity can be resolved.
       %
\end{remark}
       \begin{remark}
    Just to make the important point very clear: since our perturbations are infinitesimal, the results we obtain from the proposed method in \Cref{nsa_prefac:shadow_cont_abs} are \textbf{exact}. Our perturbation does \textbf{not} entail (additional) numeric error.
       \end{remark}
        
\begin{example}[Premisses Cannot be Neglected]
    Here is an example that should explain why we had to stipulate the second premise of \Cref{nsa_prefac:shadow_cont_abs}, the congruency of the projective shadows. Define $\Psi: [-1,1] \rightarrow \R^2$ as
        \[
                \Psi(t):= \begin{pmatrix} t \\ |t| \end{pmatrix}
\]
Although both components are continuous we cannot apply \Cref{nsa_prefac:shadow_cont_abs} and \Cref{nsa_pg:psh_normalize} because the projective shadows do note coincide. Pick an arbitrary positive infinitesimal hyperreal number $\Dt$. Then for $t_0 = 0$ it holds true:
\[
\psh (\Psi(t_0+\Dt))= \sh \left(\frac{1}{|\Dt|} \begin{pmatrix} \Dt \\ |\Dt|   \end{pmatrix} \right) =   \sh \left(\frac{1}{\Dt} \begin{pmatrix} \Dt \\\Dt   \end{pmatrix} \right) =\begin{pmatrix} 1 \\ 1   \end{pmatrix} 
\]
But for the negation of $\Dt$:\[
\psh (\Psi(t_0-\Dt))= \sh \left(\frac{1}{|-\Dt|} \begin{pmatrix} -\Dt \\ |-\Dt|   \end{pmatrix} \right) =   \sh \left(\frac{1}{\Dt} \begin{pmatrix} -\Dt \\ \Dt   \end{pmatrix} \right) =\begin{pmatrix} -1 \\ 1   \end{pmatrix} 
\]
This is not really astonishing since $t/|t|$ has a non-removable singularity at $t_0 = 0$.
\end{example}
\section{Stability of a Solution}
\label{sec:stability_of_sol}
The concept of quasi continuity was so far only developed along a predefined and only time dependent path. Unfortunately, this is not the whole picture. There are situations where one can resolve a singularity of a construction along a fixed path but as soon as the defining free or semi--free elements on which an object may depend are perturbed slightly the construction may behave unstable or even discontinuous. One particular situation is again the intersection of two circles: this time we do not investigate the tangential situations as in \Cref{sing:expl_sqrt} but rather an even more degenerate one: the unification of the two circles. 

\begin{example}[Unified Circles Intersection]
    \label{nsa_pg:expl:univ_circles}
Again, we will intersect two circles of the same radius and draw the connecting line of the two intersections. Consider two (non-degenerate) circles $C,D$ with radius $r > 0$ and disjoint midpoints $M_C$ and $M_D$. If one now moves $M_D$ along a straight segment (parameterized using the unit inverval)  passing though $M_C$ at some time $t_0$, the two circles unify and the intersections of $C$ and $D$ expand from 4 points (counting $\II$ and $\JJ$) to an infinite set. Along a linear path it is reasonable to resolve the singularity occurring at $t_0$. But if one chooses to turn at $t_0$ and move $M_D$ continuously to a different direction the intersection points will behave discontinuous. See the pictures in \Cref{nsa_pg:univ_stable} and \Cref{nsa_pg:univ_unstable} for an illustration of the situation.
\end{example}

As shown in \cite{kortenkamp2001grundlagen}, geometric constructions can be partitioned to (semi--)free and dependent elements. Then (semi--)free elements can be used to alter the configuration. One can encapsulate the movement of (semi--)free elements into a movement alone a straight line in a (possibly high--dimensional) space. We will denote the start point of the movement by $A$ and the end point by $B$. Diverging from \cite{kortenkamp2001grundlagen}, we will view $A$ and $B$ not as points in $\C^{k+1}$ but rather as projective objects in $\CP^k$. Dynamic movements can then be modeled using the function $P: [0,1] \rightarrow \CP^k$ with $v(t) = t\cdot A + (t-1) \cdot B$. For further details, we refer the reader to the article \cite{kortenkamp2001grundlagen}. We will now discuss the continuity of a construction on the neighborhood around the elements of the image of $P$.

\begin{definition}[Extended $C^0$-continuation]
    \label{nsa_pg:extd_c0_cont}
    Let $V \subset \CP^k$ be open and let $\Psi: W \subset V \rightarrow \CP^d$ continuous. If there is an open set $Y$ with $W \subsetneq Y \subset V$ and a continuous function $\hat \Psi: Y \rightarrow \CP^d$ with $\hat \Psi(w) = \Psi(w) \forall w \in W$, then we call $\Psi$ \textit{extended $\mathit{C^0}$--continuable on $Y$}
   and 
   $\hat \Phi$ the \textit{extended $\mathit{C^0}$--continuation of $\mathit{\Psi}$ on $Y$}.

   If even holds true that $Y = V$ then we call $\Psi$ \textit{extended quasi continuous on $\mathit{V}$}.
\end{definition}
An illustration of the situation can be found in \Cref{nsa_pg:ext_conti_fig}.
\begin{remark}
   It it obvious that if a function is extended $C^0$-continuable on a certain domain, then it is also $C^0$-continuable on a path inside the domain. 
\end{remark}
\begin{figure}[]
   \centering
   \includegraphics[width=.5\textwidth]{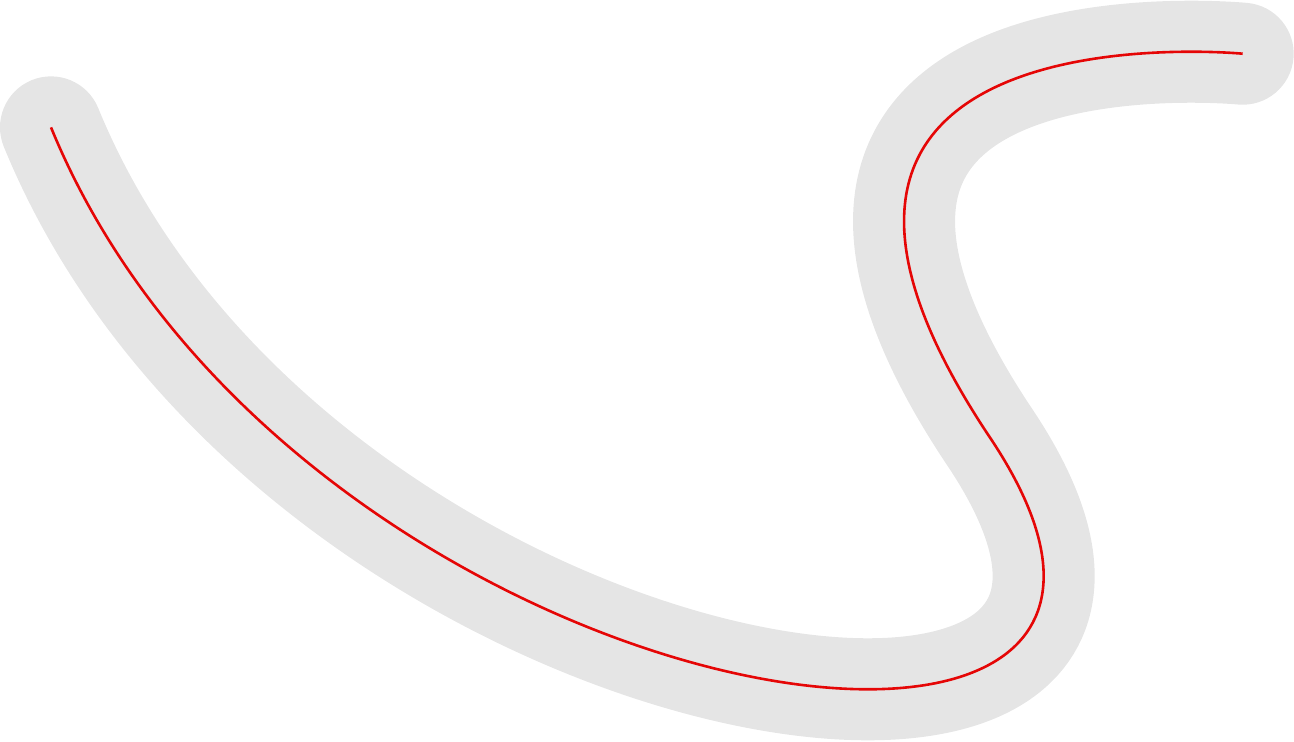}
   \caption{\textbf{The difference between a $C^0$--continuation and an extended $C^0$--continuation quasi continuity:} while the first defines continuity only along a curve (red in the picture), the extended $C^0$--continuation requires continuity in an open neighborhood (grey) around the curve. The neighborhood might be restricted by additional constraints for semi--free elements.}
   \label{nsa_pg:ext_conti_fig}
\end{figure}
\begin{remark}
    In \cite{kortenkamp2001grundlagen} it is shown that the movement of dependent elements induced by free elements along a fixed path in $\V$ can be modeled using only the time dimension. 
    
    Additionally \Cref{nsa_pg:extd_c0_cont} also ensures continuity of the movement of a dependent element if the (semi--)free elements move continuously not only depending on a time parameter, along a fixed path, but even more generally in in space. This covers also cases like in \Cref{nsa_pg:univ_stable} and \Cref{nsa_pg:univ_unstable}.
    
    Essentially $v$ models the movement of the (semi--)free elements and $\Psi$ the movement of a dependent element implied by the variation of the (semi--)free elements. The function composition $\phi = \Psi \circ v$ then models the complete movement only dependent on a time parameter.

\end{remark}
\begin{theorem}[Normalization and Extended Continous Paths]
        \label{nsa_prefac:shadow_cont_extended}
        Let $v: [0,1] \rightarrow \C^{k+1}$ be continuous, $Y$ be open with  $v([0,1]) \subset Y \subset \C^{k+1}$, $\Psi: Y \rightarrow \C^{d+1}$ a function that is continuous on $v([0,1])$, \ie $\Psi|_{v([0,1])}: v([0,1]) \rightarrow \C^{d+1}$ is a continuous function, and let $t_0 \in [0,1]$. 
        Let $\Psi$ fulfill the following premisses:
        \begin{enumerate}
            \item $\exists \epsilon \in \R, \epsilon > 0$: $\Psi(v(t))$ does not vanish on a dense subset $X$ of $B_\epsilon(v(t_0)) \cap Y$
            \item $\Psi$ can be extended on $(\overline X)^*$ such that $\exists L \in \C^{d+1} \Smz:$ 
                \begin{align*}
                    L \simeq \frac{\Psi(v(t_0) + \Dv)}{\norm{\Psi(v(t_0) + \Dv)}} \simeq \frac{\Psi(v(t_0) + \Dv')}{\norm{\Psi(v(t_0) + \Dv')}} \\ \forall \Dv, \Dv' \in \I^{k+1}: v(t_0) + \Dv, v(t_0) + \Dv' \in (\overline X)^*.
                \end{align*}
        \end{enumerate}
        Then there exists a continuous function $\Theta(t): B_\epsilon(v(t_0)) \cap Y \rightarrow \C^{d+1} \Smz$ such that
        \[
            [\Psi(v(t))] = [\Theta(v(t))]
        \]
        for all $t \in X$.       
        
This means that $[\Theta]$ is an extended $C^0$--continuation of $[\Psi]$ on the open set $\inter(\overline{B_\epsilon(v(t_0)) \cap Y})$.
\end{theorem}
\begin{proof}
    By assumption 1.\ the function 
    \[
        \Theta(v):= 
        \begin{cases}
            \dfrac{\Psi(v)}{\norm{\Psi(v)}}, &\text{if } v \in X\\
            \lim_{n\rightarrow \infty} \dfrac{\Psi(v_n)}{\norm{\Psi(v_n)}}, &\text{if } v \in \overline X  \setminus X, \text{ with } v_n \rightarrow v \; (n \rightarrow \infty)
        \end{cases}
    \]
    is well defined. By transfer we know that $\Psi$ does not vanish on $\hal(v(t_0)) \cap Y^*$ and so $\Theta$ does not vanish either. By \Cref{nsa_basics:conti}, assumption 2.\ assures the continuity of $\Theta$ at $v(t_0)$ and we can continuously extend $\Theta$ at $v(t_0)$ with $\sh(\Psi(v(t_0) + \Dv))$ with an arbitrary $\Dv$ that fulfills the requirements of assumption 2. Back in projective space, it holds true
    \[
        [\Psi(v)] = [\Theta(v)] \quad \forall v \in X.
    \]
    and therefore $[\Theta]$ is a $C^0$-continuation of $[\Psi]$ on the open set $\inter(\overline{B_\epsilon(v(t_0)) \cap Y})$.
\end{proof}
\begin{remark}
    Again we can rewrite condition 2.\ of \Cref{nsa_prefac:shadow_cont_extended} as follows:
   \[ 
       \exists L \in \CP^{d} : [L] = \psh(\Psi(v(t_0) + \Dv)) = \psh(\Psi(v(t_0) + \Dv')).
   \]
\end{remark}
\begin{figure}
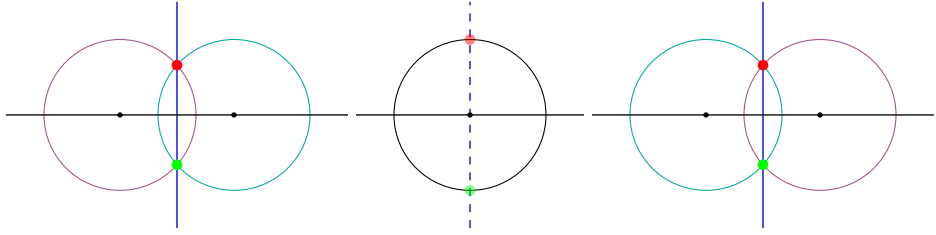

    \centering
    \stablecircles
\caption{\textbf{The Unification of Circles (semi-free)}: the connecting line of the two circle intersections is extended $C^0$-continuable if one assumes the circle centers bound to a common line.}
\label{nsa_pg:univ_stable}
\end{figure}
\begin{figure}
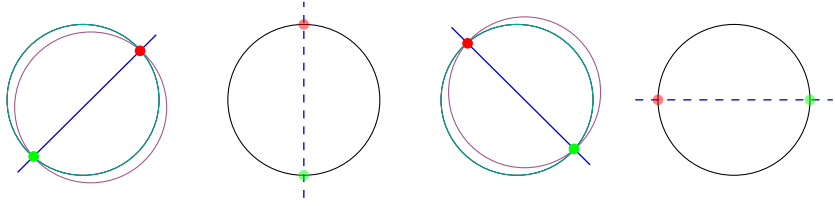

    \centering
    \unstablecircles
\caption{\textbf{The Unification of Circles (free)}: if the centers of the circles are not bound to a one dimensional subspace the connecting line of the intersection behaves discontinuously, thus the problem is not extended continuable.}
\label{nsa_pg:univ_unstable}
\end{figure}

\subsection{Application to Singularities in Geometric Constructions}
We will now describe two algorithms which will use \Cref{nsa_prefac:shadow_cont_abs} and \Cref{nsa_prefac:shadow_cont_extended} to automatically remove singularities in geometric constructions.

First we will analyze the less general $C^0$-continuation: \Cref{algo:c0-conti} takes as input a function $\Psi$ with $\Psi(t_0) = 0$ and returns a $C^0$-continuation at the given point $t_0$ if possible. The correctness of the algorithm is obvious by \Cref{nsa_prefac:shadow_cont_abs}, if the non-vanishing property of the theorem is fulfilled and as shown in ``\textsc{Grundlagen Dynamischer Geometrie}'' (\cite{kortenkamp2001grundlagen}), singularities of geometric constructions are isolated and therefore we only have to check whether the left and right side limit coincide (see \Cref{nsa_basics:sqeezing_limits}).
\begin{algorithm}[h]
    \DontPrintSemicolon
    \SetAlgoLined
    \KwIn{A continuous function $\Psi: (0,1) \rightarrow \C^{d+1}$ and $t_0 \in (0,1)$}
    \KwOut{$C^0$-Continuation of $[\Psi]$ at $t_0$}
    \Begin{
    Pick arbitrary $\Dt \in \I \Smz$\;
    $L \leftarrow \sh\left(\frac{\Psi(t_0 + \Dt)}{\norm{\Psi(t_0 + \Dt)}}\right)$\;
    $L^+ \leftarrow \frac{\Psi(t_0 + \Dt)}{\norm{\Psi(t_0 + \Dt)}}$\;
    $L^- \leftarrow \frac{\Psi(t_0 - \Dt)}{\norm{\Psi(t_0 - \Dt)}}$\;
    \If{$L =$ NaN}{
        \KwRet{ERROR: vanishing $\Psi$}
    }
    \ElseIf{$L \not \simeq L^+$ or $L \not \simeq L^-$}{
        \KwRet{ERROR: discontinuous}
    }
    \Else{
        \KwRet{$[\Psi(t_0)] = L$}
    }
    }
    \caption{$C^0$-Continuation Algorithm}
    \label{algo:c0-conti}
\end{algorithm}

\Cref{algo:c0-extended-conti} checks additionally whether the solution is stable under spatial perturbations and therefore a continuation is extended continuous. Since this is a randomized algorithm one might develop a probability estimation on how certain one can be that the solutions is actually continuous. Usually this is done using the Schwartz–Zippel lemma (see for example \cite{mitzenmacher2005probability}, or including radical expressions: \cite{tulone2000randomized}). We will defer the development of randomized continuity testing using infinitesimal deflection to future work and use \Cref{algo:c0-extended-conti} as heuristic algorithm. In practice it turns out that the algorithm quickly can decide if a given singularity $t_0$ is removable or not by only checking a few randomized inputs.
\begin{algorithm}[h]
    \DontPrintSemicolon
    \SetAlgoLined
    \KwIn{Continuous functions $v: [0,1] \rightarrow \C^{k+1}$, $\Psi: Y \rightarrow \C^{d+1}$ with open set $Y$, $v([0,1]) \subset Y \subset \C^{k+1}$ and $t_0 \in [0,1]$.}
    \KwOut{Extended $C^0$-Continuation of $[\Psi]$ at $t_0$}
    \Begin{
        Pick $n \in \N$ arbitrarily $\Dv_i \in \I^{k+1} \Smz$ ($i \in 1 \ldots n$), $\Dv_i \neq \Dv_j \, (i \neq j)$ such that $v(t_0) + \Dv_i$ is in the preimage of $\Psi$\;
        Evaluate $\Psi_i := \frac{\Psi(v(t_0) + \Dv_i)}{\norm{\Psi(v(t_0) + \Dv_i)}}$ for all $i \in \{1, \ldots, n\}$\;
        $L \leftarrow \sh(\Psi_i)$ for arbitrary $i \in \{1, \ldots, n \}$ \; 
    \If{any $\Psi_i = 0$ }{
        \KwRet{ERROR: vanishing $\Psi$}
    }
    \ElseIf{any $\Psi_i \not \simeq \Psi_j$}{
        \KwRet{ERROR: discontinuous}
    }
    \Else{
        \KwRet{$[\Psi(v(t_0))] = L$}
    }
    }
    \caption{Extended $C^0$-Continuation Algorithm}
    \label{algo:c0-extended-conti}
\end{algorithm}
\section{Numerics, Experimental Results \& Examples}
It is not easy to achieve an implementation of the hyperreal and hypercomplex numbers for real time application. One particular ansatz is to use ultrafilters to construct the field, which is highly non-constructive, since it depends on the axiom of choice (see \cite{goldblatt}). Although there are constructive methods (see \eg \cite{palmgren1995constructive,  palmgren2001unifying}), an implementation remains challenging.
One particular implementation was given by J.~Fleuriot (see for example \cite{fleuriot2012}) using higher order logic in Isabelle (\cite{nipkow2002isabelle}).

We decided to use a much smaller, but still non-Archimedean, subfield of the hyperreal and hypercomplex numbers: the so called \textit{Levi-Civita field}. The complex Levi-Civita field is also algebraically closed, which is essential for compass constructions. Especially it is the smallest non-Archimedean field which admits radical expression. For details we refer the reader to \cite{shamseddine2010analysis}.

Furthermore, we will discuss several examples like von-Staudt constructions and the disjoint circle intersection. We will give concrete numerical values to illustrate our method. And finally we will discuss the avoidance of singularities a priori. 
\subsection{The Levi-Civita Field}
The Levi-Civita field is build up by a single positive infinitesimal number, the so called number $d$.
\begin{definition}[The Number $d$]
  Pick an arbitrary but fixed positive real infinitesimal number and call it $\mathit{d}$, \ie $d \in \I_\R, d > 0$.
\end{definition}
\begin{remark}
 Note that $d^q$ is infinitesimal if and only if $q > 0$ and unlimited if and only if $q <  0$. 
\end{remark}
\begin{definition}[Left-Finite Set, \cite{shamseddine2010analysis}]
  A subset $M \subset \mathbb{Q}$ is called left-finite, if and only if for every number $r \in \mathbb{Q}$ there are only finitely many elements of M that are smaller than $r$.
\end{definition}
\begin{definition}[The Levi-Civita Field, \cite{shamseddine2010analysis}]
 \[
\RLC = \sum_{q\in M \subset \mathbb{Q}} a_q{d}^q \quad (a_q\; \in \R),\qquad
\CLC = \sum_{q\in M \subset \mathbb{Q}} a_q{d}^q \quad (a_q\; \in \C)
 \]
 where $M$ is left-finite. One can show that $(\RLC, +, \cdot)$, with the canonical polynomial operators, is a field and $\CLC$ is algebraically closed. 
\end{definition}
We will not discuss how to implement the Levi-Civita field, but the implementation is straight forward and very fast. For further details see \cite{shamseddine2010analysis}. For our sample implementation we used \textit{CindyJS} (see \cite{vonGagern2016}) and its scripting language \textit{CindyScript}, which was originally developed for \textit{Cinderella} (see \cite{richter2000user}).
\subsection{Von-Staudt Construction}
Our first example will be a von-Staudt construction. These constructions can be used to encode summation and multiplication employing geometry, and use auxiliary points on which restrictions are imposed (which are essentially removable singularities, as we will find out later). We will use the summation example to show how our developed methods can be used to resolve the singularities such that these restrictions can be omitted.
\begin{example}[Classical von-Staudt Contruction]
\label{implementation:von-staudt-orig}
We will follow \cite{richter2011perspectives} \page 89 ff.: for given points $\mathbf{x} = (x,0,1)^T$ and $\mathbf{y} = (y,0,1)^T$ we will construct the sum $\mathbf{x+y}$ according to the projective scale $\mathbf{0} = (0,0,1)^T$ and $\mathbf{\infty} = (1,0,0)^T$. The choice of the projective scale is arbitrary, but for the sake for simplicity, we use a classical setting in which $\mathbf{\infty}$ is a far point. Hence, the homogeneous coordinates of $x+y$ are equal to $\mathbf{x+y}  = (x+y,0,1)^T$, up to scalar multiplication.

   The constructions reads a follows: denote the connecting line of $\mathbf{0}$ and $\mathbf{\infty}$ by $l$. Then, choose an auxiliary point $E$, which is not incident to $l$. Choose a further auxiliary point $F$ on $\join(\mathbf{\infty}, E)$, which is unequal to $E$. Then we further proceed with the following operations:
   \begin{align*}
       G & = \meet(\join(\mathbf{0}, E), \join(\mathbf{y}, F))\\ 
       H & = \meet(\join(\mathbf{\infty}, G), \join(\mathbf{x}, E))\\ 
       m & = \join(F,H)\\
       \mathbf{x+y}  & = \meet(l, m)
   \end{align*}
   See \Cref{impl:vstaudt-nondegen} for a picture of the construction.
\begin{figure}[]
       \centering
   \begin{tikzpicture}
\draw (-0.3,0) -- (12,0);
\node at (12,0) [right= 1mm, below = 0.5mm] {${l}$};

\coordinate (z) at (0,0);
\fill (z) circle [radius=2pt];
\node at (z) [below = 2mm] {$\mathbf{0}$};

\coordinate (inf) at (11,0);
\fill (inf) circle [radius=2pt];
\node at (inf) [below = 2mm] {$\infty$};

\coordinate (x) at (2,0);
\fill (x) circle [radius=2pt];
\node at (x) [below = 3mm] {$\mathbf{x}$};

\coordinate (y) at (4,0);
\fill (y) circle [radius=2pt];
\node at (y) [below = 3mm] {$\mathbf{y}$};

\coordinate (xy) at (6,0);
\fill (xy) circle [radius=2pt];
\node at (xy) [below = 2.5mm] {$\mathbf{x+y}$};

\coordinate (E) at (2,2);
\fill (E) circle [radius=2pt];
\node at (E) [left = 3mm, above = 0.5mm] {${E}$};
\coordinate (F) at (4,2);
\fill (F) circle [radius=2pt];
\node at (F) [right = 3mm, above =1mm] {${F}$};

\coordinate (H) at (2,4);
\fill (H) circle [radius=2pt];
\node at (H) [above = 3.0mm] {${H}$};
\coordinate (G) at (4,4);
\fill (G) circle [radius=2pt];
\node at (G) [above = 3.0mm] {${G}$};

\draw [shorten <=-0.3cm, shorten >=-0.3cm] (z) -- (G);
\draw [shorten <=-0.3cm, shorten >=-0.3cm] (H) -- (xy);
\draw [shorten <=-0.3cm, shorten >=-0.3cm] (x) -- (H);
\draw [shorten <=-0.3cm, shorten >=-0.3cm] (G) -- (y);

\draw [shorten <=-0.3cm, shorten >=-3cm] (E) -- (F);
\draw (7,2) to[out=0,in=135] (inf);
\draw [start angle=0,shorten <=-0.3cm, shorten >=-3cm] (H) -- (G);
\draw (7,4) to[out=0,in=90] (inf);
\end{tikzpicture}
    \caption{non--degenerate von-Staudt}
    \label{impl:vstaudt-nondegen}
\end{figure}
\end{example}
As we can see the von-Staudt construction has two requirements:
\begin{enumerate}
    \item $E$ and $F$ have to be disjoint,
    \item the point $E$ must not be incident to $l$ (hence $F$ is also not incident to $l$)
\end{enumerate}
If one violates these constraints, the construction breaks down in standard geometry, because essential construction steps are undefined.

We will now violate both prerequisites of the construction in standard geometry. Then, we will use our method for resolving singularities by applying non-standard projective geometry and show how our method behaves numerically. We will always use the following coordinates for our construction: $\mathbf{0} = (0,0,1)^T, \mathbf{x} = (2,0,1)^T, \mathbf{y} = (4,0,1)^T, \mathbf{\infty} = (1,0,0)^T$. This should result in $\mathbf{x+y} = (6,0,1)^T \sim (1,0, \frac{1}{6})^T$.
\begin{example}[Von-Staudt Construction - non-disjoint auxiliary Points]
    Let break the first requirement: the auxiliary points $E$ and $F$ must not coincide. It is easy to see that if $E=F$, then also $H=F$ and $G=F$, just by definition. In particular, the line $m = \join(H,F)$ which defines $x+y = \meet(l,m)$ is not defined (and thus equal to $(0,0,0)^T \not \in \RP^2$). This results in $x+y$ not being defined (\ie being equal to $(0,0,0)^T \not \in \RP^2$).

    The operators $\join$ and $\meet$ are defined using the cross-product in $\RP^2$, which is a continuous function $\times: \R^3 \times \R^3 \rightarrow \R^3$. The whole construction is build up using cross product, so the construction steps are compositions of continuous functions and therefore continuous. So we can apply \Cref{nsa_prefac:shadow_cont_extended}: we infinitesimally perturb the free point $E$ and replacing the point by $E':= E + (\epsilon_x,\epsilon_y,0)^T$ and infinitesimally perturb $F$ to $F'$ with $F':= F + (\delta_x, \delta_y,0)$ such that $F' \neq E'$, but $F' \simeq E'$.
    
    The actual perturbation is arbitrary as long as it fulfills the requirements of the theorem. According to \Cref{algo:c0-extended-conti} we should actually test several perturbations to get an indicator if the function is actually continuous. Furthermore, note that we will use the appreciable cross-product of \Cref{nsa_pg:cross-prod} for all non--standard projective joins and meets.  A picture of the construction can be found in \Cref{impl:vS-merging-pts} and the computational values in \Cref{impl:vSt-merge-table}. Using the perturbation one finds 
    \[
    m' = {\color{red}d^1} \begin{pmatrix} -0.125\\ -0.125\\ 0.75 \end{pmatrix}
    \]
    especially $m'$ is infinitesimal (since $d^1$ is infinitesimal). If one would apply the shadow function it would yield $\sh (m') = (0,0,0)^T$, which resembles the degeneracy of the construction in standard geometry, but applying the projective shadow the result is
    \[
        \psh(m') = \begin{pmatrix} -0.1667\\ -0.1667\\ 1 \end{pmatrix}
    \]
        which is the correct line, hence results in the desired point: 
        \[
        \meet(l, \psh m') =  \begin{pmatrix} 1\\ 0\\ \frac{1}{6} \end{pmatrix} \sim \begin{pmatrix} 6\\ 0\\ 1 \end{pmatrix} = \mathbf{x+y}
        \]
    \begin{figure}[]
       \centering
\begin{tikzpicture}
\draw (-0.3,0) -- (12,0);
\node at (12,0) [right= 1mm, below = 0.5mm] {${l}$};

\coordinate (z) at (0,0);
\fill (z) circle [radius=2pt];
\node at (z) [below = 2mm] {$\mathbf{0}$};

\coordinate (inf) at (11,0);
\fill (inf) circle [radius=2pt];
\node at (inf) [below = 2mm] {$\infty$};

\coordinate (x) at (2,0);
\fill (x) circle [radius=2pt];
\node at (x) [below = 3mm] {$\mathbf{x}$};

\coordinate (y) at (4,0);
\fill (y) circle [radius=2pt];
\node at (y) [below = 3mm] {$\mathbf{y}$};

\coordinate (xy) at (6,0);
\fill [opacity=0.4] (xy) circle [radius=2pt];
\node at (xy) [opacity=0.4,below = 2.5mm] {$\mathbf{x+y}$};

\coordinate (F) at (4,2);
\fill (F) circle [radius=2pt];
\node at (F) [above =2mm] {${F(\simeq E \simeq H \simeq G)}$};

\draw [shorten <=-0.3cm, shorten >=-0.3cm] (z) -- (F);
\draw [dashed,shorten <=-0.3cm, shorten >=-0.3cm] (F) -- (xy);
\draw [shorten <=-0.3cm, shorten >=-0.3cm] (x) -- (F);
\draw [shorten <=-0.3cm, shorten >=-0.3cm] (F) -- (y);

\draw (F) -- (7,2);
\draw (7,2) to[out=0,in=135] (inf);
\end{tikzpicture}
    \caption{Degenerate von-Staudt: Merging Auxiliary Points}
    \label{impl:vS-merging-pts}
\end{figure}
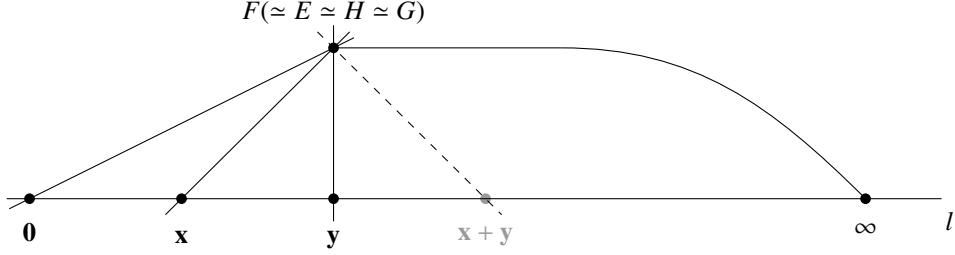
\end{example}

\begin{table}[]
\centering
\footnotesize
\begin{tabular}{cccccc}
    &$E/E'$       &$F/F'$       &$H/H'$       &$m/m'$       &$x+y, (x+y)'$  \\
CindyJS & $\begin{pmatrix}1\\ 0.5\\ 0.25 \end{pmatrix}$         & $\begin{pmatrix}1\\ 0.5\\ 0.25 \end{pmatrix}$      & $\begin{pmatrix}1\\ 0.5\\ 0.25 \end{pmatrix}$      & $\begin{pmatrix}0\\ 0\\ 0 \end{pmatrix}$       & $\begin{pmatrix}0\\ 0\\ 0 \end{pmatrix}$  \\
NSA     &  $\begin{pmatrix}1\cdot d^{0} + \epsilon_x\\ 0.5\cdot d^{0}+ \epsilon_y\\ 0.25\cdot d^{0}\end{pmatrix}$     &  $\begin{pmatrix}1\cdot d^{0} + \delta_x\\ 0.5\cdot d^{0}+ \delta_y\\ 0.25\cdot d^{0}\end{pmatrix}$     & $\begin{pmatrix}1\cdot d^{0} + \gamma_x\\ 0.5\cdot d^{0}+ \gamma_y\\ 0.25\cdot d^{0}\end{pmatrix}$      & $\begin{pmatrix}-0.125\cdot{\color{red} d^1}\\ -0.125\cdot {\color{red}d^{1}}\\ 0.75\cdot {\color{red}d^{1}} \end{pmatrix}$       &$\begin{pmatrix}-1\cdot d^{0}\\ 0\cdot d^{0}\\ -0.1667\cdot d^{0} \end{pmatrix}$  \\
$\psh$     & $\begin{pmatrix}1\\ 0.5\\ 0.25 \end{pmatrix}$      & $\begin{pmatrix}1\\ 0.5\\ 0.25 \end{pmatrix}$ & $\begin{pmatrix}1\\ 0.5\\ 0.25\end{pmatrix}$        &$\begin{pmatrix}-0.1667\\ -0.1667\\ 1 \end{pmatrix}$ & $\begin{pmatrix}1\\ 0\\ 0.1667 \end{pmatrix}$
\end{tabular}
\caption{\textbf{Merging auxiliary points}. Perturbation: $\epsilon_x = \epsilon_y = (1+d)$, $\delta_{x,y}$ and $\gamma_{x,y}$ calculated accordingly. Therefore $F' \neq H'$ but $F' \simeq H'$ (tails omitted for readability), $m' = \join(F',H')$ is properly defined and infinitesimal.}
\label{impl:vSt-merge-table}
\end{table}
\begin{remark}
    In the non-degenerate non-standard construction it was actually not important for $F$ to be incident (and not only almost incident) to $\join(\infty, E)$. By continuity of all operations and the non-degeneracy, we could actually perturb all points and generate the same result via projective shadows.

    In the degenerate case the situation is different: if $F'$ and $H'$ are only chosen to be in the projective halo of $E'$, we can generate arbitrary lines $m$ (which are almost incident to $E'$) but with arbitrarily chosen direction! 

We will shed some light on this using \Cref{nsa_prefac:shadow_cont_extended}. The perturbations are chosen to be element of $(\I^{k+1} \Smz )\cap V^*$, where $V$ denotes the subspace of a (semi-)free element, which models the restrictions a semi-free element is subjected. This means the perturbations must be chosen such that these restrictions of the dependent elements are not violated.
\end{remark}
Now we will degenerate the construction even more: we move the point $E$ onto the line $l = \join(\mathbf{0}, \mathbf{\infty})$.
\begin{example}[Von-Staudt construction -- $E$ incident to $l$]
    We break the second requirement: $E$ must not be incident to $l$. If we move $E$ to $\mathbf{x}$ then practically the whole construction breaks down: $G$, since $x = E$ and their join is undefined, $H$ because $G$ is undefined, $m$ since $H$ is undefined and of course so is $\mathbf{x+y}$. A picture of the situation can be found in \Cref{impl:vS-on-line}.

    Again we infinitesimally perturb the free point $E$ replacing the point by $E':= E + (\epsilon_x,\epsilon_y,\epsilon_z)^T$ and the dependent elements $F', G', H', m'$ and $(x+y)'$ accordingly. The values of the resulting vectors can be read up in \Cref{impl:vSt-on-line-point-table}.
\begin{table}[]
\centering
\footnotesize
\begin{tabular}{ccccccc}
    &$E/E'$       &$F/F'$ & $G/G'$       &$H/H'$       &       &$x+y, (x+y)'$  \\
CindyJS & $\begin{pmatrix}1\\ 0\\ 0.5 \end{pmatrix}$ & $\begin{pmatrix}1\\ 0\\ 0.25 \end{pmatrix}$ & $\begin{pmatrix}0\\ 0\\ 0 \end{pmatrix}$ & $\begin{pmatrix}0\\ 0\\ 0 \end{pmatrix}$ & &  $\begin{pmatrix}0\\ 0\\ 0 \end{pmatrix}$ \\
NSA & $\begin{pmatrix}1\cdot d^{0}\\ 1\cdot {\color{red} d^{1}}\\ 0.5\cdot d^{0} \end{pmatrix}$ & $\begin{pmatrix}1\cdot d^{0}\\ 0.5\cdot {\color{red} d^{1}}\\ 0.25\cdot d^{0} \end{pmatrix}$ & $\begin{pmatrix}1\cdot d^{0}\\ 1\cdot {\color{red} d^{1}}\\ 0.25\cdot d^{0} \end{pmatrix}$ & $\begin{pmatrix}-1\cdot d^{0}\\ -2\cdot {\color{red} d^{1}}\\ -0.5\cdot d^{0} \end{pmatrix}$ &  & $\begin{pmatrix}6\cdot {\color{red} d^{1}}\\ {\color{red} 0 \cdot d^{0}}\\ 1\cdot {\color{red} d^{1}} \end{pmatrix}$\\
$\psh$ & $\begin{pmatrix}1\\ 0\\ 0.5 \end{pmatrix}$ & $\begin{pmatrix}1\\ 0\\ 0.25 \end{pmatrix}$ & $\begin{pmatrix}1\\ 0\\ 0.25 \end{pmatrix}$ & $\begin{pmatrix}1\\ 0\\ 0.5 \end{pmatrix}$ &  & $\begin{pmatrix}1\\ 0\\ 0.1667 \end{pmatrix}$ \\
       &&&&&&\\
       & $\join(\infty, E/E')$ & $\join(\mathbf{0}, E/E')$ & $\join(y, F/F')$ & $\join(\infty, G/G')$ & $\join(\mathbf{x}, E/E')$ & $m/m' = \join(H/H', F/F')$ \\
CindyJS &  $\begin{pmatrix}0\\ 1\\ 0 \end{pmatrix}$ & $\begin{pmatrix}0\\ 1\\ 0 \end{pmatrix}$ & $\begin{pmatrix}0\\ 0\\ 0 \end{pmatrix}$ & $\begin{pmatrix}0\\ 0\\ 0 \end{pmatrix}$ & $\begin{pmatrix}0\\ 0\\ 0 \end{pmatrix}$ & $\begin{pmatrix}0\\ 0\\ 0 \end{pmatrix}$ \\
NSA & $\begin{pmatrix}{\color{red} 0 \cdot d^{0}}\\ -0.5\cdot d^{0}\\ 1\cdot {\color{red} d^{1}} \end{pmatrix}$ & $\begin{pmatrix}-1\cdot {\color{red} d^{1}}\\ 1\cdot d^{0}\\ {\color{red} 0 \cdot d^{0}} \end{pmatrix}$ & $\begin{pmatrix}-0.125\cdot {\color{red} d^{1}}\\ {\color{red} 0 \cdot d^{0}}\\ 0.5\cdot {\color{red} d^{1}} \end{pmatrix}$ & $\begin{pmatrix}{\color{red} 0 \cdot d^{0}}\\ -0.25\cdot d^{0}\\ 1\cdot {\color{red} d^{1}} \end{pmatrix}$ & $\begin{pmatrix}-0.5\cdot {\color{red} d^{1}}\\ 0.5\cdot {\color{red} d^{1}}\\ 1\cdot {\color{red} d^{1}} \end{pmatrix}$ & $\begin{pmatrix}0.25\cdot {\color{red} d^{1}}\\ 0.25\cdot d^{0}\\ -1.5\cdot {\color{red} d^{1}} \end{pmatrix}$ \\
$\psh$    & $\begin{pmatrix}0\\ 1\\ 0 \end{pmatrix}$ & $\begin{pmatrix}0\\ 1\\ 0 \end{pmatrix}$ & $\begin{pmatrix}-0.25\\ 0\\ 1 \end{pmatrix}$ & $\begin{pmatrix}0\\ 1\\ 0 \end{pmatrix}$ & $\begin{pmatrix}-0.5\\ 0.5\\ 1 \end{pmatrix}$ & $\begin{pmatrix}0\\ 1\\ 0 \end{pmatrix}$
\end{tabular}
\caption{\textbf{Auxiliary points on $\join(\mathbf{\mathbf{0}}, \mathbf{\infty})$}. Infinitesimal entries marked in {\color{red} red}. 
    Remember the construction: $G  = \meet(\join(\mathbf{0}, E), \join(\mathbf{y}, F)), \;
H  = \meet(\join(\mathbf{\infty}, G), \join(\mathbf{x}, E)), \;
m  = \join(F,H), \; \mathbf{x+y}   = \meet(l, m)$ with $l = (0,1,0)^T$. The $\epsilon$-tails are omitted for reasons of readability. 
}
\label{impl:vSt-on-line-point-table}
\end{table}
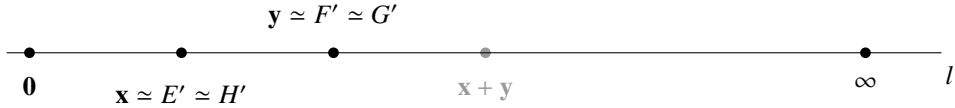
\begin{figure}[]
   \begin{tikzpicture}
\draw (-0.3,0) -- (12,0);
\node at (12,0) [right= 1mm, below = 0.5mm] {${l}$};

\coordinate (z) at (0,0);
\fill (z) circle [radius=2pt];
\node at (z) [below = 2mm] {$\mathbf{0}$};

\coordinate (inf) at (11,0);
\fill (inf) circle [radius=2pt];
\node at (inf) [below = 2mm] {$\infty$};

\coordinate (x) at (2,0);
\fill (x) circle [radius=2pt];
\node at (x) [below = 3mm] {$\mathbf{x}\simeq E' \simeq H'$};

\coordinate (y) at (4,0);
\fill (y) circle [radius=2pt];
\node at (y) [above = 2mm] {$\mathbf{y} \simeq F' \simeq G'$};

\coordinate (xy) at (6,0);
\fill [opacity=0.4] (xy) circle [radius=2pt];
\node at (xy) [opacity=0.4,below = 2.5mm] {$\mathbf{x+y}$};

\end{tikzpicture}     
\caption{Degenerate von-Staudt: auxiliary points on $l = \join(\mathbf{0}, \mathbf{\infty})$. Here $x = E = H$ and $y = F = G$.
}
    \label{impl:vS-on-line}
\end{figure}
Again, we can see that classical projective geometry over the real numbers yields a lot of undefined objects (as expected). Using non--standard analysis in projective geometry in the form of a Levi-Civita field can still provide us with the correct results. 

Also, \Cref{impl:vSt-on-line-point-table} shows some interesting effects: firstly, the application of the projective shadow when a vector is appreciable. For example the point $E' = (1 \cdot d^0, 1 \cdot d^1,0.5 \cdot d^0)^T$ (where we neglected the $\epsilon$--tail). Then the $y$--component of $(E')_y=d^1$ is discarded by the $\psh$ function, since it is one magnitude smaller than the $x$ and $z$ component. Therefore, the projective shadow and the unperturbed point $E$ coincide as desired.

But even more interesting is the line $\join(\mathbf{x}, E')= (-0.5 \cdot d^1,  0.5 \cdot d^1,  1 \cdot d^1)^T$. The vector is infinitesimal, since all entries are infinitesimal. As expected, standard projective geometry, as used in CindyJS, will yield the zero vector. If we divide the line representative by $d^1$, which is in the magnitude of the whole vector, one can find the correct result with removed singularity: $\psh(\join(\mathbf{x}, E'))= (-0.5,  0.5,  1)^T$.

The bottom line of the example is that we can resolve singularities in geometric constructions, even if they are highly degenerate, using methods of non--standard analysis. Again, we emphasize that our method does \textbf{not} introduce numerical error due to its symbolic character, while retaining a real time execution speed.
\end{example}
\subsection{Disjoint Circle Intersection}
\begin{example}[Disjoint Circle Intersection]
The ``disjoint circle intersection'' problem was introduced in \Cref{sing:expl_sqrt}. Essentially we are considering the $\join$ of two points of intersection of two disjoint circles in a tangential situation, which turns out to be removable singularity. We already analyzed the problem using non--standard arithmetic in \Cref{nsa_prefac:circ_inter_revis} from a theoretical point of view. Now we will also give computational values as well. As a quick reminder consider \Cref{impl:circ_intersect_fig} (middle).

    If one examines the connecting line of those two merging points the operation will be inherently undefined. But if one perturbs the construction infinitesimally the operation is well defined again. 

    Using non--standard analysis we can perturb the construction at any reasonable point. We chose to do so in the algorithm which intersects a line and a conic. The line $l = (\lambda, \tau, \mu)^T$ is defined as described in \cite{richter2011perspectives} \page 195 ff..
  \Cref{algo:c0-extended-conti} actually tells us to perturb the (semi-) free elements of a construction but for the sake of simplicity we assume $l$ to be free although it is dependent. This is without loss of generality since we could always choose the circle centers to generate a line with arbitrary coordinates.
    We perturb $l$ to $l'$ with $l' = (\lambda -d , \tau -d , \mu -d )^T$, where $d$ is the infinitesimal number of the Levi-Civita field. 

    The leading coefficients of $p_1$ and $p_2$ read as follows:
    \[
    p_1 = \begin{pmatrix} 1\cdot d^0 \\ 2.4495\cdot d^\frac{1}{2} \\ 1\cdot d^0 \end{pmatrix}, \quad \quad \quad p_2 = \begin{pmatrix} 1\cdot d^0 \\ - 2.4495\cdot d^\frac{1}{2}  \\ 1\cdot d^0 \end{pmatrix}
    \]
    As one can see the $y$--components of $p_1$ and $p_2$ are infinitesimal: they are the result of the square root operation in the Levi-Civita field and so they are in the magnitude of $\sqrt d$. If one takes the shadow of every component (or here equivalently the projective shadow since the vector is appreciable) one finds the double point of intersection $\psh (p_1) = \psh (p_2) = (1,0,1)^T$. Then of course the connecting line $\join (\psh (p_1), \psh (p_2))$ is the zero vector. This is essentially what standard projective geometry software would do.

    Now consider the non-standard connecting line of $p_1$ and $p_2$:
    \[
    \join(p_1, p_2) = \begin{pmatrix} 4.899\cdot d^{\frac{1}{2}} \\ 0\cdot d^{0}\\ -4.899\cdot d^{\frac{1}{2}} \end{pmatrix}
    \]
    As one can see the vector is infinitesimal, since all entries are infinitesimal. Taking the shadow of all components would therefore lead to the same result as before: $\sh(\join(p_1, p_2)) = (0,0,0)^T$. Taking the projective shadow, and normalize the vector to an appreciable length in advance one finds the desired result:
    \[
    \psh(\join(p_1, p_2)) = \sh \left( \frac{1}{ 4.899 \cdot d^\frac{1}{2}} \begin{pmatrix} 4.899\cdot d^{\frac{1}{2}} \\ 0\cdot d^{0}\\ -4.899\cdot d^{\frac{1}{2}} \end{pmatrix} \right) = 
 \begin{pmatrix} 1 \\ 0 \\ -1 \end{pmatrix} 
    \]
    The untruncated values can be found in \Cref{impl:circ_intersect_fig_table}.

    Just another word on performance: the operations over the Levi-Civita field can be performed in real time, which is essential for a dynamic geometry system. 
    \begin{table}[]
        \tiny
        \begin{tabular}{cccc}
            & $p_1$ & $p_2$ & $\join(p_1,p_2)$\\ 
        CindyJS & $\begin{pmatrix} 1\\ 0\\ 1 \end{pmatrix}$ & $\begin{pmatrix} 1\\ 0\\ 1 \end{pmatrix}$ & $\begin{pmatrix} 0\\ 0\\ 0 \end{pmatrix}$\\
        NSA & $\begin{pmatrix} 1\cdot d^{0} \cdot (1\cdot d^{0}-0.1875\cdot d^{1}+4.7076\cdot d^{1.5}+2.8125\cdot d^{2}) \\ 2.4495\cdot {\color{red} d^{0.5}} \cdot (1\cdot d^{0}+0.6124\cdot d^{0.5}-3\cdot d^{1}-1.8371\cdot d^{1.5}) \\ 1\cdot d^{0} \cdot (1\cdot d^{0}-3.1875\cdot d^{1}+1.0334\cdot d^{1.5}+1.6875\cdot d^{2}) \end{pmatrix}$ & $\begin{pmatrix} 1\cdot d^{0} \cdot (1\cdot d^{0}-0.1875\cdot d^{1}-4.7076\cdot d^{1.5}+2.8125\cdot d^{2}) \\  -2.4495\cdot {\color{red} d^{0.5}} \cdot (1\cdot d^{0}-0.6124\cdot d^{0.5}-3\cdot d^{1}+1.8371\cdot d^{1.5}) \\ 1\cdot d^{0} \cdot (1\cdot d^{0}-3.1875\cdot d^{1}-1.0334\cdot d^{1.5}+1.6875\cdot d^{2}) \end{pmatrix}$ & $\begin{pmatrix} 4.899 {\color{red} \cdot d^{0.5}} \cdot (1\cdot d^{0}) \\ {\color{red} 0\cdot d^{0}} \cdot (1\cdot d^{0}) \\ -4.899 {\color{red} \cdot d^{0.5}} \cdot (1\cdot d^{0}) \end{pmatrix}$ \\
            $\psh$ & $\begin{pmatrix} 1\\ 0\\ 1 \end{pmatrix}$ & $\begin{pmatrix} 1\\ 0\\ 1 \end{pmatrix}$ & $\begin{pmatrix}1 \\ 0 \\ -1 \end{pmatrix}$
       \end{tabular} 
       \caption{Computational values for tangential circles: points of intersection and their connecting line. For two circles $C$ and $D$, both with radius $1$, centered at the origin and at $(2,0,1)^T$ respectively. A picture of the situation can be found in \Cref{impl:circ_intersect_fig} (middle). Their points of intersection $p_1$ and $p_2$ coincide. Without methods of non--standard analysis the connecting line is undefined. The perturbed situation on the other hand yields an infinitesimal vector whose projective shadow $\psh$ is a well defined line. We marked infinitesimal entries in the leading coefficients in {\color{red} red}.}
       \label{impl:circ_intersect_fig_table}
    \end{table}
\begin{figure}
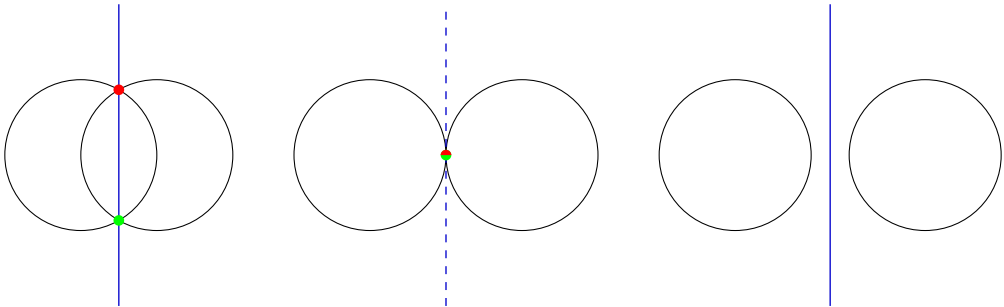

    \centering
    \drawMcirc
    \caption{$\join$ of two intersection points of circles. Left: the two real intersection points of two circles in red and green and their connecting line. Right: the so called ``powerline'': a real line connecting the (now complex) intersection points red and green. Middle: the connecting line of red and green is undefined (since they are equal) and an example of a removable singularity. Numerical values can be found in \Cref{impl:circ_intersect_fig_table}.}
    \label{impl:circ_intersect_fig}
\end{figure} 
\end{example}
\begin{example}[The Unification of Circles]
    We already encountered the intersection of identical circles in \Cref{nsa_pg:expl:univ_circles} on \page \pageref{nsa_pg:expl:univ_circles}. We will again consider the (Euclidean) points of intersection of two circles that have the same radius and (almost) the same center and consider their connecting line. When the centers coincide then the situation is singular.
    
    We employ \Cref{algo:c0-extended-conti} to check if a $C^0$-continuation is extended continuous or not. If for example the centers of the circles are bound to a common line, then the solution will be stable under infinitesimal perturbation. The unique $C^0$-continuation will yield a line that is orthogonal to the line that binds the centers and passes through the center.

    On the other hand, if the centers of the circles are free, then random infinitesimal perturbations will lead to random intersections of the circles. The connecting lines of these points are thus random lines which even could not intersect the circle at all (if we consider complex solutions). \Cref{algo:c0-extended-conti} verifies whether the results of the perturbations are infinitely close, for example by comparing the angles between the resulting vectors. The probability of a false positive result is negligible. 
    
    We will not plot any numerical values since they are essentially random noise, nevertheless see \Cref{impl:circles_unstable} for an illustration.
\begin{figure}
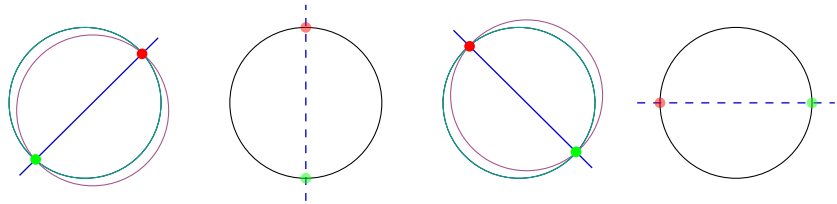

    \centering
    \unstablecircles
    \caption{\textbf{The Unification of Circles (free)}: the intersection of two identical circles yields unstable results under infinitesimal perturbation. \Cref{algo:c0-extended-conti} quickly detects that the singularity is not removable.}
\label{impl:circles_unstable}
\end{figure}
\end{example}

\section{A priori Avoidance of Singularities}
In the previous sections we discussed how to algorithmically remove singularities in geometric constructions. Although the presented methods are real time capable they still involved quite some machinery and advanced techniques and so it would be beneficial, if we did not have to employ the proposed resolutions in the first place. In the following chapter we will give a blueprint to analyze geometric algorithms and optimize their algebraic properties to remove avoidable prefactors. These prefactors are superfluous in a projective setting since they generate just another representative in the same equivalence class of an projective object. 
\label{sec:avoid_sing}
\subsection{Modeling the Avoidance of Singularities}
In the spirit of \Cref{sing_basics:ex_path}, where we split a function $\Psi: [0,1] \rightarrow C^d$ into a vanishing and a non-vanishing part in order to resolve a singularity, we will generalize this ansatz to a more abstract level. 
\begin{definition}[Removable Prefactor]
    Let $\V \subset \CP^{d_1} \times \ldots \times \CP^{d_n}$ model $n$ homogenious input elements. And let $f: \V \rightarrow  \C^{d+1}, f \not \equiv 0$ be a continuous function, which models an output element. If we can write $f$ such that there is a non-vanishing continuous function $\tilde f: \V \rightarrow \C^{d+1} \Smz$ and a function $C: \V \rightarrow \C$ with
    \[
        f(I) = C(I) \cdot \tilde f(I) \quad \forall I \in \V
    \]
    then we call $C$ a \textit{removable prefactor}.
\end{definition}
Analogously to the property of being extended $C^0$-continuous of \Cref{nsa_pg:extd_c0_cont}, we define the property of being universal $C^0$-continuous. The major difference between these concepts is that the first is only defined for one input element while the second is defined for multiple input elements.
\begin{definition}[Universal $C^0$-Continous]
    Let $\V \subset \CP^{d_1} \times \ldots \times \CP^{d_n}$ and let $f: \V \rightarrow  \CP^{d}$ be a continuous function. If there is $\hat \V \supsetneq \V$ and a continuous $\hat f: \hat \V \rightarrow \CP^d$ such that $f(I) = \hat f(I) \forall I \in \V$ then we call $f$ \textit{universal $\mathit{C^0}$-continuable} and $\hat f$ the \textit{universal $\mathit{C^0}$-continuation} on $\hat \V$.
\end{definition}
The next theorem will provide a way to find universal $C^0$-continuations if we can split a function into a non-vanishing part and a prefactor.
\begin{theorem}[Avoiding Singularities]
    Let $\V \subset \CP^{d_1} \times \ldots \times \CP^{d_n}$ and let $f: \V \rightarrow \C^{d+1}$ have a removable prefactor $C: \V \rightarrow \C$. Call $\V':=\{ v \in \V \bbar f(v) \neq 0\}$. Then $[f]: \V' \rightarrow \CP^d$ has an universal $C^0$-continuation $[\hat f]$ on $\V' \cup C^{-1}(0)$.
\end{theorem}
\begin{proof}
Since $f$ has a removable prefactor we can write 
    \[
        f(I) = C(I) \cdot \tilde f(I) \quad \forall I \in \V
    \]
    Now define: $\hat f: \V' \cup C^{-1}(0) \rightarrow \C^{d+1}, I \mapsto \tilde f(I)$
    then it holds true that 
    \[
        [f(I)] = [\hat f(I)] \forall I \in \V'.
    \]
    Furthermore, since $\tilde f$ is continuous and non-vanishing we have a universal continuation of $[f]$ on $\V' \cup C^{-1}(0)$ given by $[\tilde f]$. 
\end{proof}
\subsection{From Plücker's $\mathbf{\mu}$ to Efficient Implementation}
Plücker's $\mathbf{\mu}$ is very nice trick to construct geometric objects that fulfill multiple constraints. An example would be a line $l$ that passes through the intersection of two other lines $l_1$ and $l_2$ and additionally through a third point $q$. For further reading we refer the reader to \cite{richter2011perspectives} \page 96ff.
\begin{theorem}[Midpoint via Plücker's $\mu$]
    \label{avoid:thm:midpoint_pluecker}
    The Euclidean midpoint of two point $X$ and $Y \in \RP^2$ is given by:
   \[
        M = [Y,P_\infty]_L X + [X, P_\infty]_L Y
   \]
   where $P_\infty = \meet(\join(X , Y), l_\infty)$ and $L$ is chosen such that it is linearly independent of $X,P_\infty$ and $Y,P_\infty$.
\end{theorem}
\begin{remark}
    For objects $X,Y,Z$ in $\CP^2$ we write $[X,Y]_Z = [X,Y,Z] = \det(X,Y,Z)$.
\end{remark}
 \begin{proof}
     According to \cite{richter2011perspectives} \page~81 $X$, $Y$, the midpoint $M$ and $P_\infty$ are in harmonic position (seen from an appropriate point $L$):
     \[
         (X,Y;M, P_\infty)_L = -1
     \]
     Furthermore we know that we can construct the midpoint via Plücker's $\mu$ as a linear combination of $X$ and $Y$. This means there are $\lambda \in \C$ and $\mu \in \C$ such that $M = \lambda X + \mu Y$. In the following, we will derive those parameters using the cross-ratio property:
     \begin{align*}
         (X,Y;M, P_\infty)_L &= -1 \\
         \Leftrightarrow \frac{[X,M]_L [ Y, P_\infty]_L }{[X,P_\infty]_L [Y,M]_L} &= -1\\
         \Leftrightarrow {[X,M]_L [ Y, P_\infty]_L } &= - {[X,P_\infty]_L [Y,M]_L}\\
         \Leftrightarrow \mu [X,Y]_L [Y,P_\infty]_L &= -\lambda [X, P_\infty]_L [Y,X]_L \\
         \Leftrightarrow [X,Y]_L ( \mu [Y,P_\infty]_L &- \lambda [X, P_\infty]_L ) = 0
     \end{align*}
     An appropriate choice for $ \mu [Y,P_\infty]_L - \lambda [X, P_\infty]_L $ to be zero is $\lambda =  [Y, P_\infty]_L$ and $\mu = [X, P_\infty]_L$. And therefore we can write $M$ as  \[
         M = [Y, P_\infty]_L X + [X, P_\infty]_L Y
     \]
      which is the claim.
 \end{proof}
\begin{remark}
    An appropriate choice for $L$ in the previous theorem is $L = X \times Y$. This is fairly easy to see as we are looking for a point $L$ which is not incident to the line $l= \join(X,Y)$ since $X,Y$ and $P_\infty$ are all incident to $l$. Hence $L$, which is the line $l$ interpreted as point, is an appropriate choice since $\langle L, l \rangle = \langle L, L \rangle = \norm{L}^2 \neq 0$: therefore $L$ (interpreted as point) is not incident to $l$.
\end{remark}
\begin{lemma}[Algebra of Midpoint via Plücker's $\mu$]
    \label{avoid:lem:algebra_mid_mu}
    The algebraic representation of the midpoint $M = [Y, P_\infty, L] X + [X, P_\infty, L] Y$ of \Cref{avoid:thm:midpoint_pluecker} of $X = (x_1, x_2, x_3)$ and $Y = (y_1, y_2,y_3)$ with $L = X \times Y$ a removable prefactor  
    \[
       C := x_{2}^{2} y_{1}^{2} + x_{3}^{2} y_{1}^{2} - 2 x_{1} x_{2} y_{1} y_{2} +
x_{1}^{2} y_{2}^{2} + x_{3}^{2} y_{2}^{2} - 2 x_{1} x_{3} y_{1} y_{3} -
2 x_{2} x_{3} y_{2} y_{3} + x_{1}^{2} y_{3}^{2} + x_{2}^{2} y_{3}^{2}
    \]
\end{lemma}
\begin{proof}
    Simple calculation using a computer algebra system.
\end{proof}
\begin{theorem}[Efficient Midpoint, \cite{richter2011perspectives} \page~433]
    \label{avoid:thm:eff_mid}
    For two points $X = (x_1, x_2, x_3)^T,Y=(y_1, y_2, y_3)^T \in \RP^2$ their Euclidean midpoint $M$ is given by:
    \[
        M = y_3 X + x_3 Y.
    \]
\end{theorem}
\begin{proof}
    Using a computer algebra system we can show that the midpoint $M_\mu$ of \Cref{avoid:thm:midpoint_pluecker} can be written with $C$ of \Cref{avoid:lem:algebra_mid_mu}
   \[
       M_\mu = C \cdot (  y_3 X + x_3 Y ).
   \]
   Therefore the solutions are equivalent and the efficient formula has fewer singularities.
\end{proof}
\begin{lemma}[Removable Singulariy of \Cref{avoid:thm:midpoint_pluecker}]
    \Cref{avoid:thm:midpoint_pluecker} has a removable singularity with continuation $M = X = Y$ if $X = Y$, while \Cref{avoid:thm:eff_mid} is well defined.
\end{lemma}
\begin{proof}
    Since $X = Y$ the operation $\join(X,Y)$ will yield the zero vector. Because the definition of $M$ contains $P_\infty = \meet(\join(X , Y), l_\infty)$ it is also the zero vector. Thus, the $C$ of \Cref{avoid:lem:algebra_mid_mu} is attaining $0$ for $X=Y$.

    It is easy to see that \Cref{avoid:thm:eff_mid} is well defined for $X=Y$.
\end{proof}
\subsection{Limits of a Priori Resolution}
While it is wise to scrutinize the formulas of geometric objects for superfluous pre factors this is unfortunately not always possible. We will give a construction that has a singularity which cannot be resolved a priori. 
\begin{theorem}[No a Priori Resolution]
    There is a geometric construction whose singularities cannot be resolved a priori.
\end{theorem}
\begin{proof}
    By example: take two finite lines $l = (l_1, l_2, l_3)^T$ and $g = (g_1, g_2, g_3)^T$. Denote their point of intersection by $P = \meet(l,g)$, this operation is always well defined if $l \not = g$. A simple calculation shows that the far point $F$ of all lines perpendicular to $l$ is given by $F = (l_1, l_2, 0)^T$, which is always well defined since we assumed $l \neq l_\infty$. Then the connecting line $l'$ of $F$ and $P$ is algebraically given by
    \[
l' =        \left(- l_{2}^{2} g_{1} + l_{1} l_{2} g_{2},\,l_{1} l_{2} g_{1} - 
l_{1}^{2} g_{2},\,l_{1} l_{3} g_{1} + l_{2} l_{3} g_{2} -  l_{1}^{2}
g_{3} -  l_{2}^{2} g_{3}\right).
    \]
    Using a computer algebra system one can verify that the greatest common divisor of the components of $l'$ is $1$ and therefore there is no removable prefactor which could be canceled out in a priori.

    Nevertheless the operation $\meet(l,g)$ is undefined if $l=g$, then $l' = \join(F,P)$ is also undefined and so the operation has a singularity for $l=g$ that cannot be removed a priori. 

\end{proof}
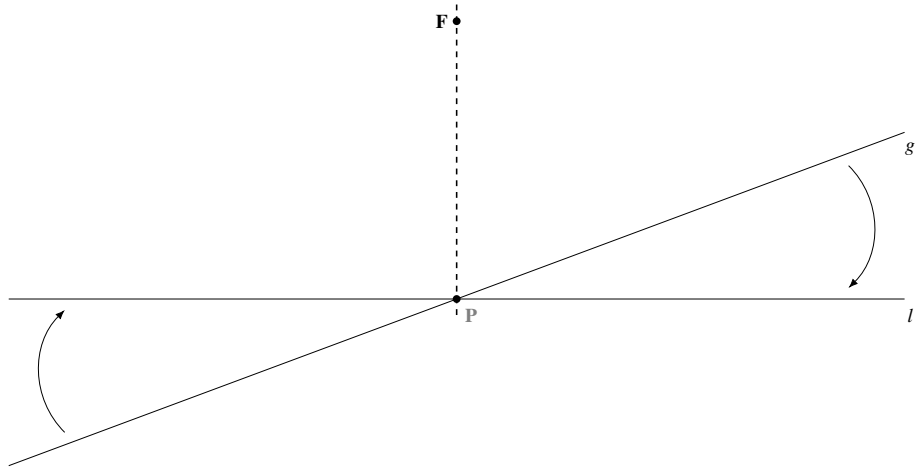
\begin{figure}
    \centering
      \usetikzlibrary{arrows.meta}
 \resizebox{0.9\linewidth}{!}{
\begin{tikzpicture}
 \draw (-8,0) -- (8,0);

 \node at (8,0) [right= 1mm, below = 0.5mm] {${l}$};

 \draw (-8,-3) -- (8,3);
 \node at (8,3) [right= 1mm, below = 0.5mm] {${g}$};

 \coordinate (z) at (0,0);
 \fill (z) circle [radius=2pt,  opacity=0.1];
 \node at (z) [below = 3mm, right = 0.3mm,  opacity=0.5] {$\mathbf{P}$};

 \coordinate (F) at (0,5);
 \fill (F) circle [radius=2pt];
 \node at (F) [above = 0mm, left = 0.3mm] {$\mathbf{F}$};

\draw [thick, dashed, shorten <=-0.3cm, shorten >=-0.3cm ] (F) -- (z);

\draw[-Latex] (7,2.4) to[out=-45,in=45] (7,0.2);
\draw[-Latex] (-7,-2.4) to[out=135,in=-135] (-7,-0.2);

\end{tikzpicture} 
}
\caption{\textbf{No a priori resolution:} the singularity in the construction if $l = g$ is not a priori removable. It depends on the position of $P$ which is determined by the position of $l$ and $g$. $P$ is not fixed, it may vary heavily on the movement of $l$ and $g$, if $P$ would be fixed, then an a priori resolution would be possible.}
\end{figure}

\bibliographystyle{elsarticle-harv}
\bibliography{main}{}
\end{document}